\documentclass[11pt,a4paper]{article}

\usepackage[british]{babel}

\usepackage[T1]{fontenc}

\usepackage[utf8]{inputenc}

\usepackage{csquotes}

\usepackage[backend=biber, maxbibnames=10, maxcitenames=4, maxalphanames=4, bibencoding=utf8, style=alphabetic, isbn=false, url=false, doi=true, giveninits=true]{biblatex}

\renewbibmacro{in:}{}
\renewbibmacro*{doi+eprint+url}{%
	\printfield{doi}%
	\newunit\newblock%
	\iftoggle{bbx:eprint}{%
		\usebibmacro{eprint}%
	}{}%
	\newunit\newblock%
	\iffieldundef{doi}{%
		\usebibmacro{url+urldate}}%
	{}%
}
\AtEveryBibitem{\clearfield{month}}
\AtEveryBibitem{\clearfield{note}}

\usepackage{setspace}

\usepackage{lmodern}

\usepackage[indentafter]{titlesec}
\titleformat{name=\section}{}{\thesection.}{0.8em}{\centering\scshape}
\titleformat{name=\subsection}[runin]{}{\thesubsection.}{0.5em}{\bfseries}[.]
\titleformat{name=\subsubsection}[runin]{}{\thesubsubsection.}{0.5em}{\itshape}[.]
\titleformat{name=\paragraph,numberless}[runin]{}{}{0em}{}[.]
\titlespacing{\paragraph}{0em}{0em}{0.5em}
\titleformat{name=\subparagraph,numberless}[runin]{}{}{0em}{}[.]
\titlespacing{\subparagraph}{0em}{0em}{0.5em}

\usepackage{xurl}
\usepackage{mathtools}

\DeclarePairedDelimiter\abs{\lvert}{\rvert}%
\DeclarePairedDelimiter\norm{\lVert}{\rVert}%

\makeatletter
\let\oldabs\abs
\def\abs{\@ifstar{\oldabs}{\oldabs*}}
\let\oldnorm\norm
\def\norm{\@ifstar{\oldnorm}{\oldnorm*}}
\makeatother

\usepackage{amssymb}

\usepackage{amsthm}

\usepackage{mathrsfs}

\usepackage{dsfont}

\usepackage{enumitem}
\setlist[enumerate]{noitemsep, partopsep=0pt, topsep=0pt, parsep=0pt, itemsep=0pt}
\setlist[itemize]{noitemsep, partopsep=0pt, topsep=0pt, parsep=0pt, itemsep=0pt}

\usepackage{interval}

\intervalconfig{soft open fences}

\usepackage[stretch=10]{microtype}

\usepackage[affil-it, noblocks]{authblk}

\usepackage[dvipsnames]{xcolor}
\usepackage[normalem]{ulem}
\usepackage{float}

\usepackage[draft=false]{hyperref}
\hypersetup{
	colorlinks = true, %
	urlcolor   = blue, %
	linkcolor  = blue, %
	citecolor  = ForestGreen %
}
\newcommand{\linkunderline}[1]{%
	{\renewcommand{\ULthickness}{0.35pt}\ULdepth=1.6pt\relax\uline{#1}}}
\let\templatehref\href
\renewcommand{\href}[2]{\templatehref{#1}{\linkunderline{#2}}}
\ExplSyntaxOn
\NewDocumentCommand{\bibbreakablehref}{mm}{%
	\begingroup
	\ttfamily
	\tl_map_inline:nn {#2} {%
		\templatehref{#1}{\linkunderline{##1}}%
		\penalty\UrlBreakPenalty
	}%
	\endgroup
}
\ExplSyntaxOff
\DeclareFieldFormat{doi}{%
	\mkbibacro{DOI}\addcolon\space
	\bibbreakablehref{https://doi.org/#1}{#1}}
\DeclareFieldFormat{url}{%
	\mkbibacro{URL}\addcolon\space
	\bibbreakablehref{#1}{#1}}
\makeatletter
\DeclareFieldFormat{eprint:arxiv}{%
	arXiv\addcolon\space
	\ifhyperref
		{\bibbreakablehref{https://arxiv.org/\abx@arxivpath/#1}{#1}%
		 \iffieldundef{eprintclass}
		 	{}
		 	{\templatehref{https://arxiv.org/\abx@arxivpath/#1}{\linkunderline{\texttt{ }}}%
		 	 \penalty\UrlBreakPenalty
		 	 \templatehref{https://arxiv.org/\abx@arxivpath/#1}{\linkunderline{\texttt{\mkbibbrackets{\thefield{eprintclass}}}}}}}
		{\texttt{#1}%
		 \iffieldundef{eprintclass}
		 	{}
		 	{\addspace\texttt{\mkbibbrackets{\thefield{eprintclass}}}}}}
\DeclareFieldAlias{eprint:arXiv}{eprint:arxiv}
\makeatother
\renewbibmacro*{cite}{%
	\printtext[bibhyperref]{%
		\linkunderline{%
			\textbf{%
				\printfield{labelprefix}%
				\printfield{labelalpha}%
				\printfield{extraalpha}%
				\ifbool{bbx:subentry}
				{\printfield{entrysetcount}}
				{}}}}}

\usepackage{aliascnt}
\usepackage[nameinlink,noabbrev,capitalize]{cleveref}
\crefname{equation}{}{}

\usepackage[plain]{fullpage}

\newlist{theoenum}{enumerate}{1} %
\setlist[theoenum]{label=\normalfont(\roman*), ref=\theproposition~\normalfont(\roman*), noitemsep, partopsep=0pt, topsep=0pt, parsep=0pt, itemsep=0pt}
\crefalias{theoenumi}{theorem}

\usepackage{tocloft}

\newcommand{\startappendix}{%
	\appendix
	\crefalias{section}{appsec}
	\titleformat{name=\section}{}{\appendixname~\thesection.}{0.8em}{\centering\scshape}
	\addtocontents{toc}{\protect\renewcommand{\protect\cftsecpresnum}{\appendixname~}}
	\addtocontents{toc}{\protect\renewcommand{\protect\cftsecaftersnum}{.\space}}
	\addtocontents{toc}{\protect\setlength{\protect\cftsecnumwidth}{6.7em}}
	\numberwithin{figure}{section}
}

\usepackage{titlefoot}

\usepackage{pgfplots}
\pgfplotsset{compat=1.18}

\usepackage[textsize=tiny]{todonotes}

\theoremstyle{plain}
\newtheorem{theorem}{Theorem}[section]
\newtheorem*{theorem*}{Theorem}

\newaliascnt{proposition}{theorem}
\newtheorem{proposition}[proposition]{Proposition}
\aliascntresetthe{proposition}
\crefname{proposition}{Proposition}{Propositions}
\Crefname{proposition}{Proposition}{Propositions}

\newaliascnt{corollary}{theorem}

\aliascntresetthe{corollary}
\crefname{corollary}{Corollary}{Corollaries}
\Crefname{corollary}{Corollary}{Corollaries}

\newaliascnt{conjecture}{theorem}

\aliascntresetthe{conjecture}
\crefname{conjecture}{Conjecture}{Conjectures}
\Crefname{conjecture}{Conjecture}{Conjectures}

\newaliascnt{lemma}{theorem}
\newtheorem{lemma}[lemma]{Lemma}
\aliascntresetthe{lemma}
\crefname{lemma}{Lemma}{Lemmas}
\Crefname{lemma}{Lemma}{Lemmas}
\newtheorem*{lemma*}{Lemma}

\theoremstyle{definition}
\newaliascnt{definition}{theorem}
\newtheorem{definition}[definition]{Definition}
\aliascntresetthe{definition}
\crefname{definition}{Definition}{Definitions}
\Crefname{definition}{Definition}{Definitions}

\theoremstyle{remark}
\newaliascnt{example}{theorem}

\aliascntresetthe{example}
\crefname{example}{Example}{Examples}
\Crefname{example}{Example}{Examples}

\newaliascnt{remark}{theorem}
\newtheorem{remark}[remark]{Remark}
\aliascntresetthe{remark}
\crefname{remark}{Remark}{Remarks}
\Crefname{remark}{Remark}{Remarks}
\newtheorem*{remark*}{Remark}

\newtheorem*{ack*}{Acknowledgements}

\crefname{theorem}{Theorem}{Theorems}
\Crefname{theorem}{Theorem}{Theorems}

\newaliascnt{appsec}{section}
\aliascntresetthe{appsec}
\crefname{appsec}{Appendix}{Appendices}
\Crefname{appsec}{Appendix}{Appendices}

\renewcommand{\epsilon}{\varepsilon}

\let\originalleft\left
\let\originalright\right
\renewcommand{\left}{\mathopen{}\mathclose\bgroup\originalleft}
\renewcommand{\right}{\aftergroup\egroup\originalright}

\title{%
  {\large\bfseries\MakeUppercase{The measure contraction property on Grushin spaces}}%
}

\author{%
  Michael Albert%
  \protect\footnotemark[1]%
  \protect\footnotemark[3]%
  \and
  Samu\"el Borza%
  \protect\footnotemark[2]%
  \protect\footnotemark[4]%
}

\date{}

\makeatletter
\def\@maketitle{%
  \newpage
  {\centering
    \@title\par
    \vskip 1.5em%
      {\large
        \lineskip .5em%
        \begin{tabular}[t]{c}%
          \@author
        \end{tabular}\par}%
    \ifx\@date\@empty\else
      \vskip 1em%
        {\large \@date}%
    \fi
    \par}%
  \vskip 1.5em%
}

\let\template@maketitle\maketitle
\renewcommand{\maketitle}{%
  \begingroup
  \renewcommand{\thefootnote}{\fnsymbol{footnote}}%
  \long\def\@makefntext##1{%
    \parindent 0pt%
    \noindent\makebox[0pt][r]{\@makefnmark\,}##1%
  }%
  \template@maketitle

  \let\orig@makefnmark\@makefnmark

  \begingroup
  \def\@makefnmark{}%
  \long\def\@makefntext##1{%
    \parindent 0pt%
    \noindent ##1%
  }%
  \footnotetext[0]{Date: \today.}%
  \endgroup

  \footnotetext[1]{%
    Analysis and PDE Unit, Okinawa Institute of Science and Technology,
    1919-1 Tancha, Onna-son, Kunigami-gun,
    Okinawa 904-0495, Japan}%
  \footnotetext[2]{%
    Faculty of Mathematics, University of Vienna,
    Oskar-Morgenstern-Platz 1, 1090 Vienna, Austria}%

  \def\@makefnmark{}%
  \footnotetext[3]{%
    \textit{E-mails}: %
    \orig@makefnmark
    \href{mailto:michael.albert@oist.jp}%
    {\nolinkurl{michael.albert@oist.jp}};\,
    {\let\@makefnmark\orig@makefnmark \footnotemark[4]}%
    \href{mailto:samuel.borza@univie.ac.at}%
    {\nolinkurl{samuel.borza@univie.ac.at}}}%
  \endgroup
}
\makeatother				

\begin{document}

\maketitle

\providecommand{\keywords}[1]
{
    \par\noindent\textbf{\textit{Keywords---}} #1\par
}

\providecommand{\msc}[1]
{
    \noindent\textbf{\textit{MSC (2020)---}} #1\par
}

\begin{abstract}
    We determine the sharp measure contraction exponents of two families of
    Grushin-type metric measure spaces. The radial Grushin space \(\mathbb{G}^{n+m}\) is \(\mathbb{R}^{n}\times\mathbb{R}^{m}\), equipped with Lebesgue measure and generated by \(X_i=\partial_{x_i}\) and \(Y_j=|x|\partial_{y_j}\), for \(1\leq i\leq n\) and \(1\leq j\leq m\).
    We prove that \(\mathbb{G}^{n+m}\) satisfies \(\operatorname{MCP}(K,N)\) if and
    only if $N\geq n+4m$ and $K\leq 0$. We also show that, for \(\alpha\geq1\), the \(\alpha\)-Grushin plane generated by \(X=\partial_x\) and \(Y_\alpha=|x|^\alpha\partial_y\) satisfies \(\operatorname{MCP}(K,N)\) if and only if \(K\leq0\) and \(N\geq N_\alpha\), where
    \[
        N_\alpha
        :=
        1+\max_{L>1}
        \frac{(2\alpha+1)L}
        {(L-1)^{2\alpha+1}+1}.
    \]
    This resolves the conjecture posed in \cite{borza2022} and, for
    integer \(\alpha\geq2\), provides the first
    examples of real-analytic sub-Riemannian structures with noninteger
    curvature exponent. Both results recover the known curvature exponent \(5\) of the classical Grushin plane, corresponding respectively to \(n=m=1\) and \(\alpha=1\).
\end{abstract}

\keywords{Measure contraction property, Grushin spaces,
    sub-Riemannian geometry}

\msc{53C17, 49Q22, 53C23}

{\renewcommand{\contentsname}{\large Contents}%
  \small
  \tableofcontents
}

\section{Introduction}
The Grushin plane is the sub-Riemannian structure on
\(\mathbb{R}^{2}\) generated by the vector fields
\begin{align*}
    X=\partial_x,
    \qquad
    Y=x\partial_y.
\end{align*}
The Grushin model originates in the work of Baouendi and Gru\v{s}in on degenerate elliptic and hypoelliptic second-order operators of the form \(X^2+Y^2=\partial_x^2+x^2\partial_y^2\) \cite{Baouendi1967,Grushin1970}. The Grushin plane has become a standard model of a
rank-varying, non-equiregular sub-Riemannian geometry. The Grushin plane falls into the class of two-dimensional \emph{almost-Riemannian manifolds} \cite[Chapter 9]{AgrachevBarilariBoscainBook2020}. There is no shortage of work on the Grushin plane and its variants, including models obtained by modifying the generating vector fields, the underlying space, the bracket-generating step, or the reference measure.

An early
study of the control metrics associated with degenerate elliptic operators
appears in \cite{FranchiLanconelli1984}. Weighted Sobolev and Hardy
inequalities, unique continuation, regularity, and semilinear equations for
Grushin-type and related degenerate operators were investigated in
\cite{FranchiGutierrezWheeden1994,Garofalo1993,GarofaloShen1994,
    DAmbrosio2004,DomokosManfredi2010,KogojLanconelli2012,
    MontiMorbidelli2006}. Isoperimetric questions and their relation to
Heisenberg geometry were studied in
\cite{MontiMorbidelli2004,ArcozziBaldi2008,FranceschiMonti2016}.
Metric, embedding, and quasiconformal properties were considered in
\cite{WuJang-Mei2015,WuEmbedding2015,RomneyVellis2017,
    gartland2017quasiconformalmappingsgrushinplane}. Nonlinear and viscosity
equations in Grushin-type spaces were studied in
\cite{BieskeGong2006,Bieske2007Comparison,Bieske2009InfiniteHarmonic},
while heat kernels, stochastic processes, self-adjointness, and quantum
confinement were investigated in
\cite{Yutian2011,changli2014,BoscainNeel2020,BoscainPrandi2016,
    BoscainLaurent2013,GalloneMichelangeliPozzoli2019}. Geodesics and optimal
synthesis for Grushin-type models were studied in
\cite{chang04,ChangLi2012,borza2022,
    albert2025geodesicsgrushinspaces,
    albert2026optimalsynthesisradiallysymmetric}. For the broader theory of
two- and three-dimensional almost-Riemannian structures, see
\cite{AgrachevBoscainSigalotti2008,BoscainCharlotGayeMason2015}.

In this work, we study two Grushin-type metric measure spaces. The first is the \emph{\(\alpha\)-Grushin plane} \(\bigl(\mathbb G_\alpha^2,d_{CC},\mathcal L^2\bigr)\), where \(\alpha\geq1\), whose underlying space is \(\mathbb R^2\) and whose Carnot--Carathéodory distance is induced by
\[
    X=\partial_x,
    \qquad
    Y_\alpha=|x|^\alpha\partial_y.
\]
When \(\alpha\in\mathbb N\), the smooth Hörmander bracket-generating family \(\{\partial_x,x^\alpha\partial_y\}\) generates the same horizontal curves with the same lengths as \(\{X,Y_\alpha\}\). For
noninteger \(\alpha\), the generating fields have lower regularity along
\(\{x=0\}\), but this causes no difficulty for the constructions and
results used in this paper. The classical Grushin plane corresponds to the case $\alpha = 1$.

We also study the \emph{radial Grushin space}
\(\mathbb G^{n+m}\), whose underlying metric measure space is
\((\mathbb R^{n+m},d_{CC},\mathcal L^{n+m})\). Writing points as
\((x,y)\in\mathbb R^n\times\mathbb R^m\), the distance is induced by the
vector fields
\[
    X_i=\partial_{x_i},
    \qquad
    Y_j=|x|\partial_{y_j},
    \qquad
    i=1,\ldots,n,\quad j=1,\ldots,m.
\]
Although the fields \(Y_j\) are not differentiable at \(x=0\), this
apparent loss of regularity can be removed by replacing them with the
smooth redundant family \(Y_{ij}=x_i\partial_{y_j}\), where
\(i=1,\ldots,n\) and \(j=1,\ldots,m\). The two generating families
determine the same horizontal curves and assign them the same lengths,
and therefore induce the same Carnot--Carathéodory distance. The latter
family is smooth and Hörmander bracket-generating of step two, so
\(\mathbb G^{n+m}\) is a smooth rank-varying sub-Riemannian structure.
The term \emph{radial Grushin space}, which we adopt here, is not
standard.

On a Riemannian manifold, Ricci curvature is defined using the
Levi--Civita connection. On a sub-Riemannian manifold, the absence of a
canonical connection makes curvature more difficult to study. One
approach is to consider the curvature-dimension condition
\(\operatorname{CD}(K,N)\), introduced through optimal transport by
Lott--Villani and Sturm
\cite{LottVillani2009,Sturm2006I,Sturm2006II}, and the measure
contraction property \(\operatorname{MCP}(K,N)\), introduced by Ohta and
Sturm \cite{ohta2007,Sturm2006II}. These synthetic conditions are
designed to replace, in the setting of metric measure spaces, the
Riemannian bounds \(\operatorname{Ric}\geq K\) and \(\dim M\leq N\).
The condition \(\operatorname{CD}(K,N)\), which is stronger than
\(\operatorname{MCP}(K,N)\), is never satisfied by genuinely sub-Riemannian manifolds
\cite{juillet2009,juillet2020,MagnaboscoRossi2023,rizzi2023failure,NavarroPan2025} (see however \cite{BorzaTashiro2024}).

By contrast, the measure contraction property
\(\operatorname{MCP}(K,N)\) is satisfied in many sub-Riemannian,
sub-Finsler, and Sasakian structures
\cite{agrachev2014generalizedriccicurvaturebounds,juilletphdthesis,
    badreddine2018measurecontractionpropertiestwostep,rizzi2016,
    BaloghKristalySipos2014,PaulWYLee2016Discrete,Zhang2025}. The first sub-Riemannian result is due to Juillet, who proved
that the Heisenberg group satisfies \(\operatorname{MCP}(K,N)\) if and
only if \(K\leq0\) and \(N\geq5\)
\cite{juilletphdthesis,juillet2009}. More generally, compact analytic two-step sub-Riemannian manifolds equipped with a
smooth measure, as well as Lipschitz Carnot groups, satisfy a measure
contraction property
\cite{badreddine2018measurecontractionpropertiestwostep}. Consequently,
every step-two Carnot group satisfies \(\operatorname{MCP}(0,N)\) for
some finite \(N\), although the sharp exponent remains unknown in most
cases. For the classical Grushin plane,
\(\operatorname{MCP}(K,N)\) holds precisely when \(K\leq0\) and
\(N\geq5\) \cite{BarilariRizzi2019}. Nevertheless, many
sub-Riemannian structures do not satisfy any measure contraction bound.
The first such examples are the Martinet and Engel structures, which fail
\(\operatorname{MCP}(K,N)\) for every \(K\in\mathbb R\) and every
finite \(N\); the same holds for any Carnot group admitting either of
these structures as a quotient \cite{borzaMCPfailure}.
As explained in \cref{sec:preliminaries}, it is enough to consider
\(\operatorname{MCP}(0,N)\) for both spaces studied here. The infimum
of the admissible values of \(N\) is called the \emph{curvature exponent}.

Our first main result determines the curvature exponent of the radial
Grushin space.
\begin{theorem}\label{Radial Grushin MCP Theorem}
    The radial Grushin space $(\mathbb{G}^{n+m},d_{CC},\mathcal{L}^{n+m})$ satisfies $\operatorname{MCP}(K,N)$ if and only if $N\geq n+4m$ and $K\leq 0$.
\end{theorem}

Our second main result confirms the value conjectured in
\cite{borza2022} for the curvature exponent of the
$\alpha$-Grushin plane.
\begin{theorem}\label{alpha Grushin MCP Theorem}
    Let $\alpha\geq 1$ and define
    \begin{equation}
        \label{eq:Nalpha-definition}
        N_\alpha:=
        1+\max_{L>1}
        \frac{(2\alpha+1)L}{(L-1)^{2\alpha+1}+1}.
    \end{equation}
    The $\alpha$-Grushin plane $(\mathbb{G}^2_\alpha,d_{CC},\mathcal{L}^2)$ satisfies $\operatorname{MCP}(K,N)$ if and only if $K\leq 0$ and $N\geq N_\alpha$.
\end{theorem}

The value of \(N_\alpha\) in \cref{eq:Nalpha-definition} is exactly
that conjectured in \cite{borza2022}, see
\cref{subsec:equivalent-curvature-exponent}. The dependence of the curvature exponent on \(\alpha\) is illustrated in
\cref{fig:alpha-curvature-exponent}.
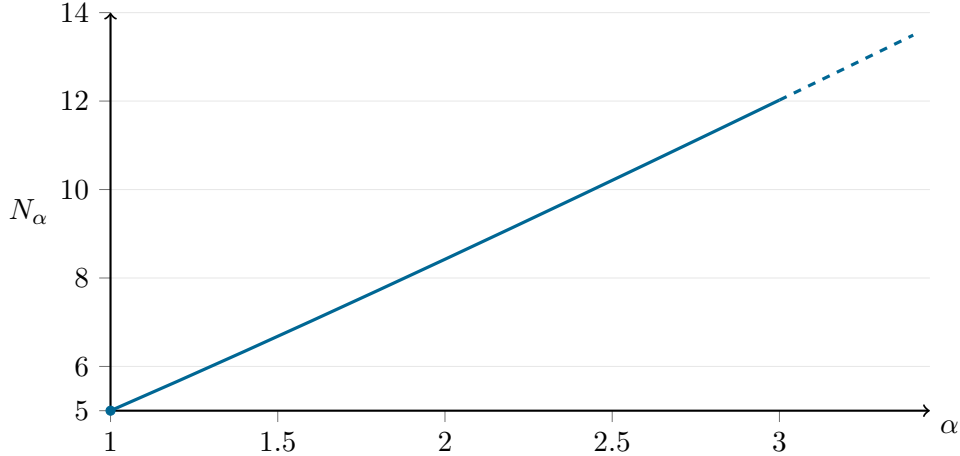
\begin{figure}[H]
    \centering
    \begin{tikzpicture}
        \begin{axis}[
                width=0.78\textwidth,
                height=0.43\textwidth,
                xmin=1,
                xmax=3.45,
                ymin=5,
                ymax=14,
                axis lines=left,
                axis line style={->,line width=0.8pt},
                xlabel={\(\alpha\)},
                ylabel={\(N_\alpha\)},
                xtick={1,1.5,2,2.5,3},
                ytick={5,6,8,10,12,14},
                tick align=outside,
                tick style={black!55},
                ymajorgrids,
                major grid style={black!9},
                clip=false,
                every axis x label/.style={
                        at={(current axis.right of origin)},
                        anchor=north west
                    },
                every axis y label/.style={
                        at={(axis description cs:-0.10,0.5)},
                        anchor=center,
                        rotate=0
                    }
            ]
            \addplot[
                MidnightBlue,
                very thick
            ] coordinates {
                    (1.0,5.00000000)
                    (1.2,5.66351569)
                    (1.4,6.33936084)
                    (1.6,7.02557795)
                    (1.8,7.72063670)
                    (2.0,8.42331834)
                    (2.2,9.13263678)
                    (2.4,9.84778310)
                    (2.6,10.56808575)
                    (2.8,11.29298124)
                    (3.0,12.02199229)
                };
            \addplot[
                MidnightBlue,
                very thick,
                dashed
            ] coordinates {
                    (3.0,12.02199229)
                    (3.2,12.75471117)
                    (3.4,13.49078693)
                };
            \addplot[
                MidnightBlue,
                only marks,
                mark=*,
                mark size=1.7pt
            ] coordinates {(1,5)};
        \end{axis}
    \end{tikzpicture}
    \caption{The curvature exponent \(N_\alpha\) of the
        \(\alpha\)-Grushin plane as a function of \(\alpha\geq1\). It obeys
        \(N_\alpha=4\alpha+2-\log(4\alpha+2)+o(1)\) as
        \(\alpha\to\infty\) (see \cref{rem:alpha-curvature-asymptotic}).}
    \label{fig:alpha-curvature-exponent}
\end{figure}

The proofs of both theorems rely on the Jacobian characterization of
the measure contraction property established in
\cite{BarilariRizzi2019}. In either of the two spaces considered above,
fix \(q_0\) and let \(q_1\notin\operatorname{Cut}(q_0)\). Let
\(\lambda_0\) be the initial
covector of the unique minimizing geodesic from \(q_0\) to \(q_1\).
The corresponding distortion coefficient is
\[
    \beta_t(q_0,q_1)
    =
    \frac{J_{q_0}(t;\lambda_0)}
    {J_{q_0}(1;\lambda_0)},
\]
where \(J_{q_0}(t;\lambda_0)\) is the Jacobian determinant of the
exponential map at time \(t\). Thus, determining the curvature exponent
reduces to finding the smallest \(N\) for which
\begin{equation}\label{eq:intro-distortion-bound}
    \beta_t(q_0,q_1)\geq t^N
\end{equation}
holds for every \(t\in[0,1]\), every base point \(q_0\), and every
\(q_1\notin\operatorname{Cut}(q_0)\).
Accordingly, the core analytic arguments in
\cref{sec:radial-grushin-mcp,sec:alpha-grushin-mcp} establish
\cref{eq:intro-distortion-bound} for the sharp values of \(N\) stated in
\cref{Radial Grushin MCP Theorem,alpha Grushin MCP Theorem},
respectively. For the \(\alpha\)-Grushin plane, we introduce a new
method combining a first-contact argument for the logarithmic
derivative of the Jacobian with a blow-up analysis. The method
establishes the sharp lower bound for the curvature exponent.
Techniques capable of doing so are rare in sub-Riemannian geometry, and
we hope that this one can be adapted to other settings. When
\(\alpha=1\), it also gives a new and more natural proof of the result
in \cite{BarilariRizzi2019} that the curvature exponent of the classical
Grushin plane is \(5\).

We close this introduction with two important observations on the study of the
measure contraction property in sub-Riemannian geometry. First, the curvature exponent \(N_\alpha\) is irrational for every integer
\(\alpha\geq2\) (see \cref{rem:alpha-curvature-irrationality}); for example, \(N_2\approx 8.4233\). Since for \(\alpha \in \mathbb{N}\) the
\(\alpha\)-Grushin plane is generated by the real-analytic
bracket-generating family \(\{\partial_x,x^\alpha\partial_y\}\), these planes therefore provide the very first examples of real-analytic sub-Riemannian structures with a finite, noninteger curvature exponent. Noninteger
curvature exponents were previously known in sub-Finsler geometry
\cite{BorzaMagnaboscoRossiTashiro2025CurvatureExponent} but never observed in sub-Riemannian structures. Second, it is
conjectured in \cite[Conjecture~18]{Zhang2025} that a Carnot group with noninteger curvature
exponent should exist. Since the classical Grushin plane is a quotient of the Heisenberg group
and both spaces have curvature exponent \(5\), it is natural to ask
whether the Carnot group admitting \(\mathbb G_\alpha^2\) as a quotient
might likewise have curvature exponent \(N_\alpha\) and provide
the example predicted by the conjecture. For integer
\(\alpha\), the space \(\mathbb G_\alpha^2\) is a quotient of the model filiform
Carnot group of step \(\alpha+1\)
\cite[Section~3.1]{Ottazzi2025}. For \(\alpha\geq2\), however, the filiform group upstairs fails every
measure contraction property \cite[Corollary~1.10]{borzaMCPfailure}. This natural
quotient construction cannot produce a Carnot group with finite
noninteger curvature exponent.

\begin{ack*}
    S.B.'s research was supported by \textbf{SUBLOR}, a research project funded by the European Union under the Horizon Europe programme's Marie Skłodowska-Curie Actions Postdoctoral Fellowships, grant agreement No.\@ \href{https://cordis.europa.eu/project/id/101282277}{101282277}.

    \textbf{AI disclosure.} Several key steps in the proof of \cref{alpha Grushin MCP Theorem} were suggested by OpenAI's GPT-5.6 Thinking model which was accessed during May and June of 2026. The authors take full responsibility for the accuracy of all arguments presented in the paper.

    \textbf{Data availability.} No datasets were generated or analysed during this study.
\end{ack*}

\section{Preliminaries on sub-Riemannian geometry}\label{sec:preliminaries}

We briefly review important notions in sub-Riemannian geometry. For more details, we refer the reader to \cite{AgrachevBarilariBoscain2012}. Let \(M\) be a smooth connected manifold of dimension \(d\), and let
\(X_1,\ldots,X_k\) be globally defined locally Lipschitz fields on \(M\).  Denote by $\mathcal{F}:\mathbb{R}^k\times M\to TM$
the map $\mathcal{F}(u,q)=\sum_{j=1}^{k}u_jX_j(q)$,
and write \(\mathcal{F}_q=\mathcal{F}(\cdot,q)\). An absolutely continuous curve
\(\gamma:[0,T]\rightarrow M\), where absolute continuity is understood
in local coordinates, is \emph{admissible} if there exists a control
\(\mathbf{u}\in L^1([0,T],\mathbb{R}^k)\) such that
\begin{align*}
    \dot{\gamma}(t)
    =
    \mathcal{F}(\mathbf{u}(t),\gamma(t))
    =
    \sum_{j=1}^k\mathbf{u}_j(t)X_j(\gamma(t)),
    \qquad \text{for a.e. }t\in[0,T].
\end{align*}
In the appendix in
\cite[Chapter~3]{AgrachevBarilariBoscainBook2020}, it is proven that the pointwise
minimal control
\begin{align*}
    \mathbf{u}^*(t)
    :=
    \operatorname*{argmin}_{
        u\in
        \mathcal{F}_{\gamma(t)}^{-1}(\dot{\gamma}(t))
    }
    \lvert u\rvert^2
\end{align*}
is defined for almost every \(t\in[0,T]\) and belongs to
\(L^1([0,T],\mathbb{R}^k)\). Furthermore, every nonconstant admissible curve admits an arc-length reparametrization. For this parametrization, its minimal control satisfies \(\lvert\mathbf u^*(t)\rvert=1\) almost everywhere and therefore belongs to \(L^\infty\). The \emph{length} of an admissible curve
\(\gamma:[0,T]\rightarrow M\) is $\ell(\gamma)
    :=
    \lVert\mathbf{u}^*\rVert_{L^1([0,T],\mathbb{R}^k)}$.     The \emph{Carnot--Carath\'eodory distance} is
\begin{align*}
    d_{CC}(x,y)
    =
    \inf\left\{
    \ell(\gamma):
    \begin{array}{l}
        T>0,\ \gamma:[0,T]\rightarrow M\text{ admissible},\
        \gamma(0)=x,\ \gamma(T)=y
    \end{array}
    \right\}.
\end{align*}
When the generating fields are smooth, they satisfy the
\emph{bracket-generating condition}, also called the \emph{H\"ormander
    condition}, if
\begin{equation*}
    \operatorname{span}\bigl\{Y(q):
    Y\in\operatorname{Lie}(X_1,\ldots,X_k)\bigr\}
    =T_qM
    \qquad\text{for every }q\in M,
\end{equation*}
where \(\operatorname{Lie}(X_1,\ldots,X_k)\) is the Lie algebra generated
by the fields and their iterated Lie brackets. By the Chow--Rashevskii theorem, this condition ensures that \(d_{CC}\) is finite and induces the
manifold topology. An admissible curve whose length equals \(d_{CC}(x,y)\) is called a
\emph{length minimizer}. An admissible curve is a \emph{geodesic} if
it is locally length minimizing. We call the
sub-Riemannian manifold $(M,d_{CC})$ geodesic if there is a length minimizer connecting all pairs of points in $M$.  When $d_{CC}$ induces the manifold topology, $(M,d_{CC})$ is locally
compact and the metric Hopf--Rinow theorem implies that every
complete sub-Riemannian manifold equipped with its Carnot--Carathéodory distance is proper and geodesic.

For each \(\lambda\in T_q^*M\), define the fiber-linear functions
\(h_j:T^*M\to\mathbb R\) by
$h_j(\lambda):=\langle\lambda,X_j(q)\rangle$ for $j=1,\ldots,k$.
For \(\nu\in\{0,1\}\), define the control Hamiltonian
\(\mathcal H_\nu:T^*M\times\mathbb R^k\to\mathbb R\) by
\begin{equation*}
    \mathcal H_\nu(\lambda,u)
    :=
    \left\langle\lambda,
    \mathcal F(u,\pi(\lambda))
    \right\rangle
    -
    \frac{\nu}{2}\lvert u\rvert^2
    =
    \sum_{j=1}^k u_jh_j(\lambda)
    -
    \frac{\nu}{2}\lvert u\rvert^2,
\end{equation*}
where \(\pi:T^*M\to M\) is the canonical projection. The
Hamiltonian \(H:T^*M\to[0,\infty)\) is
    \begin{equation*}
        H(\lambda)
        :=
        \frac12\sum_{j=1}^k h_j(\lambda)^2.
    \end{equation*}

    For every \(C^1\) function \(f:T^*M\to\mathbb R\), we denote by \(\overrightarrow f\)
    its Hamiltonian vector field, determined by
    \(\sigma(\cdot,\overrightarrow f)=df\), where \(\sigma\) is the canonical symplectic
    form on \(T^*M\). When the generating fields are \(C^1\), the function
    \(\mathcal H_\nu(\cdot,u)\) is \(C^1\) for each fixed
    \(u\in\mathbb R^k\), so \(\overrightarrow{\mathcal H}_\nu(\cdot,u)\) is well defined.
    If \(H\) is \(C^2\), then \(\overrightarrow H\) is \(C^1\) and hence generates a
    locally unique Hamiltonian flow.

    A constant-speed length minimizer has a minimal control of constant norm,
    which therefore belongs to \(L^\infty([0,T],\mathbb R^k)\). We recall the
    Pontryagin maximum principle in the standard form for constant-speed
    sub-Riemannian length minimizers; see
    \cite[Theorem~3.59]{AgrachevBarilariBoscainBook2020}.

    \begin{theorem}[Pontryagin maximum principle]
        Let \(M\) be a smooth manifold equipped with \(C^1\) vector fields
        \(X_1,\ldots,X_k\), and assume that the associated Hamiltonian \(H\) is
        \(C^2\). Let \(\gamma:[0,T]\to M\) be a constant-speed length
        minimizer and let \(\mathbf u^*\) be its minimal control. Then there
        exist an absolutely continuous curve \(\lambda:[0,T]\to T^*M\) and a
        multiplier \(\nu\in\{0,1\}\) such that
        \(\pi(\lambda(t))=\gamma(t)\) and the following conditions hold:
        \begin{enumerate}[label=\normalfont(\roman*)]
            \item \textbf{(Nontriviality)}
                  \((\nu,\lambda(t))\neq(0,0)\) for every \(t\in[0,T]\);
            \item \textbf{(Adjoint equation)} For almost every
                  \(t\in[0,T]\),
                  \begin{equation*}
                      \dot\lambda(t)
                      =
                      \overrightarrow{\mathcal H}_\nu
                      \bigl(\lambda(t),\mathbf u^*(t)\bigr);
                  \end{equation*}
            \item \textbf{(Maximization condition)} For almost every
                  \(t\in[0,T]\),
                  \begin{equation*}
                      \mathcal H_\nu
                      \bigl(\lambda(t),\mathbf u^*(t)\bigr)
                      =
                      \max_{u\in\mathbb R^k}
                      \mathcal H_\nu(\lambda(t),u).
                  \end{equation*}
        \end{enumerate}
        An extremal lift satisfying these conditions with \(\nu=1\) is called
        \emph{normal}, whereas one satisfying them with \(\nu=0\) is called
        \emph{abnormal}.
    \end{theorem}

    In the normal case, the maximization condition gives
    $u_j^*(t)=h_j(\lambda(t))$ for $1\leq j\leq k$, and the adjoint equation reduces to the \emph{Hamiltonian equation}
    \begin{equation}
        \label{eq:hamiltonequation}
        \dot\lambda=\overrightarrow H(\lambda).
    \end{equation}
    In the abnormal case, it must hold that
    $h_j(\lambda(t))=0$ for every $j=1,\ldots,k$ and almost every
    $t\in[0,T]$. Thus an abnormal lift annihilates the distribution along its
    projected curve.
    The projection of a normal trajectory is always a geodesic. A sub-Riemannian manifold is called \emph{ideal} if it admits no nonconstant abnormal length minimizers.

    For \(q_0\in M\), the normal exponential map at time \(t\) is defined for
    every \(\lambda_0\in T^*_{q_0}M\) such that \(t\) lies in the maximal
    interval of existence of the solution of \cref{eq:hamiltonequation}
    starting at \(\lambda_0\), by
    \begin{align*}
        \mathrm{exp}_{q_0}^{t}(\lambda_0)
        :=
        \pi\left(e^{t\overrightarrow H}(\lambda_0)\right).
    \end{align*}
    When \(t=1\), we omit the superscript and write
    \(\mathrm{exp}_{q_0}\).
    Whenever \(H(\lambda_0)>0\), the curve
    \(t\mapsto\mathrm{exp}_{q_0}^{t}(\lambda_0)\) has constant speed
    \(\sqrt{2H(\lambda_0)}\). Since \(H\) is fiberwise quadratic, the
    exponential map satisfies
    $\mathrm{exp}_{q_0}^{t}(a\lambda_0)
        =
        \mathrm{exp}_{q_0}^{at}(\lambda_0)$ for $a>0$
    whenever both sides are defined.

    The \emph{cut time} of the normal geodesic starting at
    \(q_0\in M\) with initial covector \(\lambda_0\in T^*_{q_0}M\) is
    \begin{equation*}
        t_{\operatorname{cut}}(\lambda_0)
        :=
        \sup\left\{
        T>0:
        t\mapsto\mathrm{exp}_{q_0}(t \lambda_0)
        \text{ is length minimizing on }[0,T]
        \right\}.
    \end{equation*}
    If \(\gamma(t)=\mathrm{exp}_{q_0}(t\lambda_0)\), we also use the
    shorthand
    \(t_{\operatorname{cut}}(\gamma):=t_{\operatorname{cut}}(\lambda_0)\).
    When \(t_{\operatorname{cut}}(\lambda_0)<\infty\) and the exponential
    map is defined at that time, the corresponding \emph{cut point} is $\mathrm{exp}_{q_0}(t_{\operatorname{cut}}(\lambda_0)\lambda_0)$. The \emph{cut locus} of \(q_0\) is
    \[
        \operatorname{Cut}(q_0)
        :=
        \left\{
        \mathrm{exp}_{q_0}(t_{\operatorname{cut}}(\lambda_0)\lambda_0):
        0<t_{\operatorname{cut}}(\lambda_0)<\infty
        \right\}.
    \]
    If \(M\) is a complete and ideal sub-Riemannian manifold, then every
    point in \(M\setminus\operatorname{Cut}(q_0)\) is joined to \(q_0\) by a
    unique minimizing normal geodesic. The collection of all arclength geodesics starting from \(q_0\), together with their cut times, is called the \emph{optimal synthesis} from \(q_0\).

    After choosing local coordinates and fixed linear coordinates on
    \(T_{q_0}^*M\), define the signed Jacobian determinant by
    \begin{equation*}
        J_{q_0}(t;\lambda_0)
        :=
        \det D_{\lambda_0}\mathrm{exp}_{q_0}^t.
    \end{equation*}
    A time \(t>0\) is a \emph{conjugate time} along the normal trajectory
    generated by \(\lambda_0\) if and only if \(J_{q_0}(t;\lambda_0)=0\). The point
    \(\mathrm{exp}_{q_0}^t(\lambda_0)\) is then called a \emph{conjugate point}
    to \(q_0\) along this normal trajectory. For every \(t\) before the first conjugate time along the normal trajectory generated by \(\lambda_0\), the map \(\mathrm{exp}_{q_0}^t\) is a local diffeomorphism at \(\lambda_0\), and its Jacobian determinant there is nonzero.

    The \emph{Extended Hadamard Technique} provides a standard method for verifying a candidate cut time and cut locus. We state it for a Riemannian initial point, that is, a point \(q_0\in M\) satisfying \(\mathcal F_{q_0}(\mathbb R^k)=T_{q_0}M\).

    \begin{theorem}[Section~13.4 in
        \cite{AgrachevBarilariBoscainBook2020}]
    \label{Extended Hadamard Technique}
    Let \(M\) be a complete, ideal sub-Riemannian manifold generated by smooth vector fields \(X_1,\ldots,X_k\), and let \(q_0\in M\) be a Riemannian point. Let
    \(\operatorname{Cut}^{\ast}(q_0)\subseteq M\) be a candidate cut locus
    and, for every
    \(\lambda_0\in T^*_{q_0}M\cap H^{-1}(1/2)\), let
    \(t_{\operatorname{cut}}^{\ast}(\lambda_0)\in(0,+\infty]\) be a
candidate cut time.
Assume that $t_{\operatorname{cut}}(\lambda_0)
    \leq
    t_{\operatorname{cut}}^{\ast}(\lambda_0)$ and let \(N\subseteq T^*_{q_0}M\) be the set of covectors for which the
corresponding normal geodesics are conjectured to be minimizing up to
time \(1\):
\begin{equation*}
    N
    :=
    \left\{
    t\lambda_0:
    \lambda_0\in T^*_{q_0}M\cap H^{-1}(1/2),\
    t\in[0,t_{\operatorname{cut}}^{\ast}(\lambda_0))
    \right\}.
\end{equation*}
Assume that \(N\) satisfies the following conditions:
\begin{enumerate}[label=\normalfont(\roman*)]
    \item
          \(\mathrm{exp}_{q_0}(N)
          =
          M\setminus\operatorname{Cut}^{\ast}(q_0)\);

    \item
          the restriction
          \(\mathrm{exp}_{q_0}|_N:
          N\longrightarrow
          M\setminus\operatorname{Cut}^{\ast}(q_0)\)
          is proper and has invertible differential at every point of
          \(N\);

    \item
          \(\mathrm{exp}_{q_0}(N)\) is simply connected.
\end{enumerate}
Then \(\mathrm{exp}_{q_0}|_N\) is a diffeomorphism onto
\(M\setminus\operatorname{Cut}^{\ast}(q_0)\), and the candidate cut
times and cut locus are exact:
\begin{equation*}
    t_{\operatorname{cut}}(\lambda_0)
    =
    t_{\operatorname{cut}}^{\ast}(\lambda_0),
    \qquad
    \operatorname{Cut}(q_0)
    =
    \operatorname{Cut}^{\ast}(q_0).
\end{equation*}
\end{theorem}

We conclude this section with a brief discussion of the measure contraction property. A full formulation of synthetic curvature-dimension conditions requires the language of optimal transport; however, in the settings considered here, we can work directly with \emph{distortion coefficients}.

\begin{definition}
    Let \((\mathcal X,d,\mathfrak m)\) be a geodesic metric measure space
    with \(\mathfrak m\) a Radon measure. For \(x,y\in\mathcal X\) and
    \(t\in[0,1]\), define
    \begin{align*}
        \beta_t(x,y)
        :=
        \limsup_{r\to0^+}
        \frac{\mathfrak m\bigl(Z_t(x,B(y,r))\bigr)}
        {\mathfrak m\bigl(B(y,r)\bigr)},
    \end{align*}
    where \(Z_t(A,B)\) denotes the set of \(t\)-intermediate points of
    minimizing geodesics joining \(A\) to \(B\); explicitly,
    \begin{align*}
        Z_t(A,B)
        :=
        \left\{
        \gamma(t):
        \gamma(0)\in A,\
        \gamma(1)\in B,\
        \gamma\text{ is minimizing on }[0,1]
        \right\}.
    \end{align*}
\end{definition}

Let \(M\) be a smooth manifold equipped with a complete, ideal
sub-Riemannian structure generated by \(C^1\) vector fields
\(X_1,\ldots,X_k\), and suppose that the associated Hamiltonian \(H\)
is \(C^2\). Let \(\mathfrak m\) be a smooth positive measure on \(M\).
If \(q_1\notin\operatorname{Cut}(q_0)\), let
\(\lambda_0\in T_{q_0}^*M\) be the initial covector of the unique
minimizing normal geodesic from \(q_0\) to \(q_1\), parametrized on
\([0,1]\). Writing \(d\mathfrak m(q)=\rho(q)\,dq\) in local coordinates,
the distortion coefficient is given by
\begin{align*}
    \beta_t(q_0,q_1)
    =
    \frac{
        \rho\bigl(\mathrm{exp}_{q_0}^t(\lambda_0)\bigr)
        |J_{q_0}(t;\lambda_0)|
    }{
        \rho(q_1)
        |J_{q_0}(1;\lambda_0)|
    }.
\end{align*}
This identity follows from
\cite[Lemma~44 and equation~(37)]{BarilariRizzi2019}.
For \(0<t\leq1\), since there are no conjugate times before the cut time,
\(J_{q_0}(t;\lambda_0)\) and \(J_{q_0}(1;\lambda_0)\) have the same sign.
In the spaces considered here, \(\mathfrak m\) is Lebesgue measure, so
\(\rho\equiv1\) and the formula reduces to
\begin{align*}
    \beta_t(q_0,q_1)
    =
    \frac{J_{q_0}(t;\lambda_0)}
    {J_{q_0}(1;\lambda_0)}.
\end{align*}

\begin{theorem}[{\cite[Theorem~9]{BarilariRizzi2019}}]
    \label{MCP theorem Rizzi}
    Let $M$ be a complete, ideal sub-Riemannian manifold with smooth measure $\mathfrak{m}$. The following are equivalent.
    \begin{enumerate}[label=\normalfont(\roman*)]
        \item $\beta_t(q_0,q_1)\geq t^N$ for all $q_0\in M$ and $q_1\notin \operatorname{Cut}(q_0)$ and $t\in [0,1]$.
        \item $(M,d_{CC},\mathfrak{m})$ satisfies $\operatorname{MCP}(0,N)$;
              that is, $\mathfrak{m}(Z_t(q_0,B))\geq
                  t^N\mathfrak{m}(B)$ for every \(q_0\in M\), every Borel
              set \(B\subseteq M\), and every \(t\in[0,1]\).
    \end{enumerate}
\end{theorem}
To apply this equivalence to the \(\alpha\)-Grushin plane for arbitrary
real \(\alpha\geq1\), we use the lower-regularity Jacobian criterion in
\cref{lem:jacobian-mcp}, whose hypotheses are verified in
\cref{app:first}.

By the MCP version of the Bonnet--Myers
theorem, a metric measure space satisfying \(\operatorname{MCP}(K,N)\)
with \(K>0\) must be compact. Moreover, \(\operatorname{MCP}(0,N)\) implies
\(\operatorname{MCP}(K,N)\) for every \(K\leq0\), see
\cite[Theorem~4.3]{ohta2007}. If a metric measure space \((\mathcal X,d,\mathfrak m)\) admits global bijective dilations \(\delta_\varepsilon\), for arbitrarily small \(\varepsilon>0\), such that
\[
    d(\delta_\varepsilon x,\delta_\varepsilon y)=\varepsilon d(x,y),
    \qquad
    \mathfrak m(\delta_\varepsilon A)=\varepsilon^Q\mathfrak m(A)
\]for some \(Q>0\) and every measurable \(A\subseteq\mathcal X\), then \(\operatorname{MCP}(K,N)\) for \(K<0\) implies, after rescaling by \(\delta_\varepsilon\), \(\operatorname{MCP}(\varepsilon^2K,N)\). Letting
\(\varepsilon\to0\) yields \(\operatorname{MCP}(0,N)\) (see also \cite[Step~3 in the proof of Theorem~3]{BarilariRizzi2018}). Thus, for an unbounded metric measure space admitting such dilations, it is enough to consider \(\operatorname{MCP}(0,N)\), and the smallest \(N\geq1\) for which it holds is called the \emph{curvature exponent} of \((\mathcal X,d,\mathfrak m)\). This applies to both spaces considered here, whose dilations are \((x,y)\mapsto(\varepsilon x,\varepsilon^2y)\) and \((x,y)\mapsto(\varepsilon x,\varepsilon^{1+\alpha}y)\) for the radial Grushin space and the \(\alpha\)-Grushin plane, respectively.

\section{The measure contraction property of the radial Grushin spaces}
\label{sec:radial-grushin-mcp}

In this section, we study the sub-Riemannian structure
\(\mathbb{G}^{n+m}\) on \(\mathbb{R}^{n+m}\), generated by the vector fields
\begin{equation*}
    X_i = \partial_{x_i}
    \qquad (1\leq i\leq n), \qquad
    Y_j = |x|\partial_{y_j}
    \qquad (1\leq j\leq m),
\end{equation*}
and equipped with Lebesgue measure.
Although the fields \(Y_j\) are globally Lipschitz but not smooth along
\(\{x=0\}\), the original family and the smooth redundant family
\(X_i=\partial_{x_i}\), \(Y_{ij}=x_i\partial_{y_j}\) define the same
horizontal curves, assign the same length to every horizontal curve,
and induce the same Carnot--Carath\'eodory distance and sub-Riemannian
Hamiltonian; hence they are equivalent presentations of the same
sub-Riemannian structure. Since the latter family is smooth and
bracket-generating, it realizes this structure as a smooth rank-varying
sub-Riemannian structure. It is complete and ideal:
its only abnormal minimizing curves are constant. We refer to the spaces \(\mathbb{G}^{n+m}\) as \emph{radial Grushin spaces}. In this section, we characterize the normal geodesics of $\mathbb{G}^{n+m}$ and determine their cut and conjugate times, compute the associated distortion coefficients, and identify the curvature exponent.

At a point \((x,y)\in\mathbb{G}^{n+m}\), write a covector
\(\lambda\in T^*_{(x,y)}\mathbb{G}^{n+m}\) as $\lambda
    =
    \sum_{i=1}^{n}u_i\,dx_i
    +
    \sum_{j=1}^{m}w_j\,dy_j$,
where \((u,w)\in\mathbb{R}^{n}\times\mathbb{R}^{m}\) are the associated
fiber coordinates. The corresponding sub-Riemannian Hamiltonian is given by
\[
    H(x,y,u,w)
    =
    \frac{1}{2}\left(|u|^2+|x|^2|w|^2\right).
\]
The Hamiltonian equation \cref{eq:hamiltonequation} is then the system
\begin{equation*}
    \dot{x}=u,\qquad
    \dot{u}=-|w|^2x,\qquad
    \dot{y}=|x|^2w,\qquad
    \dot{w}=0.
\end{equation*}
Subject to the initial conditions \(x(0)=x_0\), \(y(0)=y_0\),
\(u(0)=u_0\), and \(w(0)=w_0\), we have \(w(t)\equiv w_0\).
Setting \(\omega=|w_0|\), the solution for \(x\) is
\begin{equation}\label{radial x formula}
    x(t;u_0,w_0)
    =
    \begin{cases}
        \displaystyle
        \frac{\sin(\omega t)}{\omega}u_0
        +\cos(\omega t)x_0,
         & \omega\neq0, \\[6pt]
        x_0+tu_0,
         & \omega=0.
    \end{cases}
\end{equation}
Since \(\dot y=|x|^2w_0\), we may write
\(y(t;u_0,w_0)=y_0+g(t;u_0,\omega)w_0\) when \(\omega\neq0\), while
\(y(t;u_0,0)=y_0\). If $\kappa \geq 0$ is defined so that
\(2\kappa^2=|u_0|^2+\omega^2|x_0|^2\), then
\begin{equation}\label{g formula}
    g(t;u_0,\omega)
    :=
    \frac{\kappa^2t}{\omega^2}
    +\frac{
        \omega^2|x_0|^2-|u_0|^2
    }{2\omega^3}
    \sin(\omega t)\cos(\omega t)
    +\frac{\langle x_0,u_0\rangle}{\omega^2}
    \sin^2(\omega t).
\end{equation}
For fixed \(t>0\), we have
\(\mathrm{exp}_{(x_0,y_0)}^t(u_0,w_0)
=(x(t;u_0,w_0),y(t;u_0,w_0))\). In the Cartesian cotangent
coordinates \((u_0,w_0)\), its differential has the block form
\begin{align*}
    D\mathrm{exp}_{(x_0,y_0)}^t
    =
    \begin{pmatrix}
        D_{u_0}x & D_{w_0}x \\
        D_{u_0}y & D_{w_0}y
    \end{pmatrix}.
\end{align*}
For \(\omega\neq0\), set
\begin{align*}
    h(t;u_0,\omega)
    :=
    \frac{
        \omega t\cos(\omega t)-\sin(\omega t)
    }{\omega^3}u_0
    -
    \frac{t\sin(\omega t)}{\omega}x_0.
\end{align*}
Writing \(I_p\) for the \(p\times p\) identity matrix, direct
differentiation gives
\begin{align*}
    D_{u_0}x
     & =
    \frac{\sin(\omega t)}{\omega}I_n,\qquad
    D_{w_0}x
    =
    hw_0^T, \qquad D_{u_0}y
    =
    w_0(\nabla_{u_0}g)^T,\qquad
    D_{w_0}y=
    gI_m+w_0(\nabla_{w_0}g)^T,
\end{align*}
where
\begin{align*}
    \nabla_{u_0}g
    =
    \left(
    \frac{t}{\omega^2}
    -
    \frac{\sin(\omega t)\cos(\omega t)}{\omega^3}
    \right)u_0
    +
    \frac{\sin^2(\omega t)}{\omega^2}x_0,
    \qquad
    \nabla_{w_0}g
    =
    \frac{\partial_\omega g}{\omega}w_0.
\end{align*}

To compute the Jacobian determinant
\(J_{(x_0,y_0)}(t;u_0,w_0)\), we use two standard
linear-algebraic identities; see \cite{HornJohnsonMatrixAnalysis}.
If \(A\) is invertible and the remaining blocks have compatible
dimensions, the Schur complement formula gives
\begin{equation}\label{Schur}
    \det
    \begin{pmatrix}
        A & B \\
        C & D
    \end{pmatrix}
    =
    \det A\,\det(D-CA^{-1}B).
\end{equation}
Furthermore, if \(A\) is an invertible \(N\times N\) matrix and
\(v,w\in\mathbb{R}^N\), the rank-one determinant formula gives
\begin{equation}\label{Rank-One}
    \det(A+wv^T)
    =
    \bigl(1+\langle v,A^{-1}w\rangle\bigr)\det A.
\end{equation}

Suppose that \(\omega t\notin\pi\mathbb{Z}\). By the Schur complement
formula \eqref{Schur},
\begin{align*}
    J_{(x_0,y_0)}(t;u_0,w_0)
     & =
    \left(
    \frac{\sin(\omega t)}{\omega}
    \right)^n
    \det\left(
    D_{w_0}y
    -
    \frac{\omega}{\sin(\omega t)}
    D_{u_0}yD_{w_0}x
    \right) \\
     & =
    \left(
    \frac{\sin(\omega t)}{\omega}
    \right)^n
    \det\left(
    gI_m
    +
    w_0
    \left(
    \nabla_{w_0}g
    -
    \frac{\omega}{\sin(\omega t)}
    \langle\nabla_{u_0}g,h\rangle w_0
    \right)^T
    \right).
\end{align*}
Applying the rank-one determinant formula \eqref{Rank-One} and using
\(|w_0|=\omega\), we obtain
\begin{align}\label{det pre simplified}
    J_{(x_0,y_0)}(t;u_0,w_0)
    =
    \left(
    \frac{\sin(\omega t)}{\omega}
    \right)^n
    g^{m-1}
    \left(
    g+\omega\partial_\omega g
    -
    \frac{\omega^3}{\sin(\omega t)}
    \langle\nabla_{u_0}g,h\rangle
    \right).
\end{align}

We next simplify the final factor in \cref{det pre simplified}.

\begin{lemma}\label{technical}
    Let \(\omega\neq0\). If
    \(P(t;u_0,\omega)
    :=g+\omega\partial_\omega g
    -\frac{\omega^3}{\sin(\omega t)}
    \langle\nabla_{u_0}g,h\rangle\),
    then
    \begin{align}\label{P simplification}
        P(t;u_0,\omega)
        =
        \frac{t}{\omega^2}
        \left(
        \omega^2
        \left|
        x_0+\frac{t}{2}u_0
        \right|^2
        +
        \left(
        1-\omega t\cot(\omega t)
        -\frac{\omega^2t^2}{4}
        \right)|u_0|^2
        \right).
    \end{align}
    Moreover,
    \begin{align}\label{non-negative factor}
        1-\omega t\cot(\omega t)
        -\frac{\omega^2t^2}{4}
        \geq0,
        \qquad
        0\leq t<\frac{\pi}{\omega}.
    \end{align}
    The inequality is strict for \(0<t<\pi/\omega\).
\end{lemma}

\begin{proof}
    Substituting the formulas for \(g\), \(\nabla_{u_0}g\), and \(h\)
    into the definition of \(P\) and collecting the coefficients of
    \(|u_0|^2\), \(|x_0|^2\), and \(\langle x_0,u_0\rangle\) gives
    \cref{P simplification}.

    It remains to prove the asserted positivity. With \(q(s):=1-s\cot s-s^2/4\), the function
    \[
        F(s)
        :=
        s-\frac{2\sin s\cos s}{1+\cos^2s}
    \]
    satisfies \(F(0)=0\) and
    \[
        F'(s)
        =
        \frac{(1-\cos^2s)(3-\cos^2s)}
        {(1+\cos^2s)^2}
        \geq0.
    \]
    Consequently,
    \begin{equation}\label{second comparison}
        2\sin s\cos s
        \leq
        s(1+\cos^2s),
        \qquad 0\leq s<\pi.
    \end{equation}
    Since
    \[
        2\sin^2s\,q'(s)
        =
        s(1+\cos^2s)-2\sin s\cos s,
    \]
    we have \(q'(s)\geq0\). As \(q(0)=0\), this proves
    \cref{non-negative factor}; strictness follows in the same way for
    \(s>0\).
\end{proof}

The preceding computations determine the Jacobian for
\(\omega\neq0\). We now include the case \(\omega=0\) and analyze the
conjugate times.

\begin{theorem}[Jacobian and conjugate times]\label{Conjugacy Profile}
    Setting \(\omega=|w_0|\), we have
    \begin{equation}\label{radial Jacobian formula}
        J_{(x_0,y_0)}(t;u_0,w_0)
        =
        \begin{cases}
            \displaystyle
            \left(
            \frac{\sin(\omega t)}{\omega}
            \right)^n
            g(t;u_0,\omega)^{m-1}
            P(t;u_0,\omega),
             & \omega\neq0, \\[8pt]
            \displaystyle
            t^n
            \left(
            t|x_0|^2
            +t^2\langle x_0,u_0\rangle
            +\frac{t^3}{3}|u_0|^2
            \right)^m,
             & \omega=0,
        \end{cases}
    \end{equation}
    where \(P\) is given by \cref{P simplification} and \(g\) by
    \cref{g formula}. When \(\omega\neq0\) and
    \(\omega t\in\pi\mathbb{Z}\), the first expression is understood by
    continuous extension. Moreover, for a nonconstant geodesic
    \(\gamma(t)=\mathrm{exp}_{(x_0,y_0)}^t(u_0,w_0)\), the following
    hold:
    \begin{enumerate}[label=\normalfont(\roman*)]
        \item If \(n\geq2\), \(x_0\neq0\), and \(\omega\neq0\), then
              \(t=\pi/\omega\) is the first conjugate time.

        \item If \(n=1\), \(x_0\neq0\), and \(\omega\neq0\), then
              \(t=\pi/\omega\) is conjugate if and only if \(u_0=0\). In the
              latter case it is the first conjugate time.

        \item If \(\omega=0\), then \(\gamma\) is a line and has no
              conjugate times.
    \end{enumerate}
\end{theorem}

\begin{proof}
    For \(\omega\neq0\) and \(\omega t\notin\pi\mathbb{Z}\), the formula
    follows from \cref{det pre simplified} and the definition of \(P\);
    its extension to the excluded times follows from the continuity of the
    differential. The formula for \(\omega=0\) follows by direct
    differentiation. Suppose first that \(\omega\neq0\) and
    \(0<t<\pi/\omega\). Then
    \(g(t;u_0,\omega)=\int_0^t|x(s;u_0,w_0)|^2\,ds\), and hence
    \(g>0\) for every nonconstant geodesic. By
    \cref{P simplification} and \cref{technical}, one also has
    \(P>0\). It follows that the Jacobian does not vanish for
    \(0<t<\pi/\omega\), so there are no conjugate times in this interval.

    Multiplying \cref{P simplification} by
    \(\sin(\omega t)/\omega\) and letting
    \(t\uparrow\pi/\omega\) gives
    \begin{equation}\label{endpoint P limit}
        \lim_{t\uparrow\pi/\omega}
        \frac{\sin(\omega t)}{\omega}P(t;u_0,\omega)
        =
        \frac{\pi^2}{\omega^4}|u_0|^2.
    \end{equation}
    Moreover, \cref{g formula} gives
    \(g(\pi/\omega;u_0,\omega)
    =\pi
    (|u_0|^2+\omega^2|x_0|^2)/(2\omega^3)>0\).
    Factoring one power of \(\sin(\omega t)/\omega\) in
    \cref{det pre simplified} and using
    \cref{endpoint P limit}, together with the continuity of the
    differential, shows that its determinant vanishes at
    \(t=\pi/\omega\) when \(n\geq2\). When \(n=1\), it vanishes there if
    and only if \(u_0=0\). Since there are no earlier conjugate times, this
    proves \normalfont(i) and \normalfont(ii).

    Finally, when \(\omega=0\),
    the factor appearing in the corresponding case of
    \cref{radial Jacobian formula} is
    \(\int_0^t|x_0+su_0|^2\,ds\). For \(t>0\), this is positive unless
    \(x_0=u_0=0\), which is the constant case. Thus a nonconstant geodesic
    with \(\omega=0\) has no conjugate times.
\end{proof}

We can now determine the optimal synthesis, that is, characterize the
minimizing geodesics and their cut times, using the Extended Hadamard
Technique; see
\cite{albert2025geodesicsgrushinspaces,
    albert2026optimalsynthesisradiallysymmetric,borza2022,
    AgrachevBarilariBoscainBook2020}.

\begin{theorem}\label{radial cut time theorem}
    Let \((x_0,y_0)\in\mathbb{G}^{n+m}\), and let
    \(\gamma(\cdot;u_0,w_0)\) be a nonconstant geodesic starting from
    \((x_0,y_0)\), with \(\omega=|w_0|\). If \(\omega=0\), then
    \(\gamma\) is a horizontal line and has infinite cut time. If
    \(\omega\neq0\), then $t_{\operatorname{cut}}(\gamma)=\pi/\omega$.
    Moreover, the cut locus of $(x_0, y_0)$ is given by
    \begin{equation}\label{radial cut locus formula}
        \operatorname{Cut}(x_0,y_0)
        =
        \begin{cases}
            \left\{
            (-x_0,y):
            |y-y_0|\geq\frac{\pi}{2}|x_0|^2
            \right\},
             & x_0\neq0, \\[4pt]
            \left\{
            (0,y):y\neq y_0
            \right\},
             & x_0=0.
        \end{cases}
    \end{equation}
\end{theorem}

\begin{proof}
    If \(\omega=0\), then \(\gamma(t)=(x_0+tu_0,y_0)\) and any horizontal competitor
    \(\eta=(\eta_x,\eta_y)\) with the same endpoints satisfies
    \(\ell(\eta)\geq\int_0^t|\dot\eta_x(s)|\,ds
    \geq
    |\eta_x(t)-\eta_x(0)|
    =
    |(x_0+tu_0)-x_0|
    =
    t|u_0|=\ell(\gamma|_{[0,t]})\), so \(\gamma\) minimizes for all time.

    Suppose that \(\omega\neq0\) and \(x_0\neq0\), and let \(\mathcal C\)
    denote the first set in \cref{radial cut locus formula}. By the
    discussion at the beginning of this section, \(\mathbb G^{n+m}\) is
    complete and ideal, and \((x_0,y_0)\) is a Riemannian point. For the
    candidate cut time \(\pi/\omega\), the candidate injectivity domain in
    \cref{Extended Hadamard Technique} is
    \(\mathcal N=\{(u,w):|w|<\pi\}\). For \(|w|=\pi\),
    \cref{radial x formula,g formula} give
    \begin{equation}\label{boundary exponential formula}
        \mathrm{exp}_{(x_0,y_0)}(u,w)
        =
        \left(
        -x_0,
        y_0+\frac{|u|^2+\pi^2|x_0|^2}{2\pi^2}w
        \right).
    \end{equation}
    Hence \(\partial\mathcal N\) maps onto \(\mathcal C\). If \(u\neq0\),
    choose \(\widetilde u\neq u\) with \(|\widetilde u|=|u|\).
    By \cref{boundary exponential formula}, the two corresponding geodesics
    have the same endpoint and length at \(t=1\), while
    \cref{radial x formula} shows that they are distinct. This excludes minimality past \(t=1\). If \(u=0\),
    \cref{Conjugacy Profile} shows that \(t=1\) is the first conjugate
    time. Thus \(\pi/\omega\) is an upper bound for the cut time, while
    \cref{Conjugacy Profile} also shows that
    \(D\mathrm{exp}_{(x_0,y_0)}\) is nonsingular on \(\mathcal N\).

    We next determine the image of \(\mathcal N\) by solving
    \(\mathrm{exp}_{(x_0,y_0)}(u,w)=(x_1,y_1)\) for
    \((u,w)\in\mathcal N\). Given \((x_1,y_1)\) with \(y_1\neq y_0\), set
    \(e=(y_1-y_0)/|y_1-y_0|\). Writing \(w=\rho e\),
    \(0<\rho<\pi\), the equation for the \(x\)-component gives
    \begin{equation*}
        u(\rho)
        =
        \frac{\rho}{\sin\rho}
        \bigl(x_1-\cos\rho\,x_0\bigr).
    \end{equation*}
    The remaining endpoint equation is
    \(|y_1-y_0|=\Phi_{x_1}(\rho)\), where
    \(\Phi_{x_1}(\rho):=\rho g(1;u(\rho),\rho)\). A direct computation using \cref{radial x formula} and the definition
    of \(P\) shows that
    \[
        \Phi_{x_1}'(\rho)
        =
        P(1;u(\rho),\rho)>0,
        \qquad 0<\rho<\pi,
    \]
    where the inequality follows from \cref{technical}.
    Furthermore, \cref{g formula} yields
    \begin{equation*}
        \lim_{\rho\downarrow0}\Phi_{x_1}(\rho)=0,
        \qquad
        \lim_{\rho\uparrow\pi}\Phi_{x_1}(\rho)
        =
        \begin{cases}
            +\infty,               & x_1\neq-x_0, \\[2pt]
            \dfrac{\pi}{2}|x_0|^2, & x_1=-x_0.
        \end{cases}
    \end{equation*}
    The case \(y_1=y_0\) corresponds to \(w=0\). Consequently,
    \(\mathrm{exp}_{(x_0,y_0)}(\mathcal N)
    =\mathbb G^{n+m}\setminus\mathcal C\). The same formulas and limits
    show that the preimage of every compact
    \(K\Subset\mathbb G^{n+m}\setminus\mathcal C\) is bounded away from
    \(|w|=\pi\) and bounded in \((u,w)\). Since it is closed, it is compact;
    hence the restricted exponential map is proper.

    Finally, \(\mathbb G^{n+m}\setminus\mathcal C\) is contractible: the
    homotopy
    \(F_s(x,y)=(x,y_0+(1-s)(y-y_0))\), \(0\leq s\leq1\), retracts it onto
    \(\mathbb R^n\times\{y_0\}\). Thus
    \cref{Extended Hadamard Technique} shows that
    \(\mathrm{exp}_{(x_0,y_0)}\) maps \(\mathcal N\)
    diffeomorphically onto \(\mathbb G^{n+m}\setminus\mathcal C\), proving
    the asserted cut time and cut locus.

    It remains to consider \(x_0=0\). By
    \cref{radial x formula,g formula}, \(\gamma\) lies in the affine plane
    \(\operatorname{span}\{u_0\}
    \times(y_0+\operatorname{span}\{w_0\})\). Set
    \(v=u_0/|u_0|\) and \(e=w_0/\omega\), fix \(0<t<\pi/\omega\), and let
    \(\eta=(x,y)\) be a horizontal curve joining \((0,y_0)\) to
    \(\gamma(t)\). Using its minimal controls, write
    \(\dot x=a\) and \(\dot y=|x|b\), so that
    \(\ell(\eta)=\int_0^t (|a|^2+|b|^2)^{1/2}\,d\tau\). Then by the Cauchy--Schwarz inequality,
    \(r=|x|\) and \(s=\langle y-y_0,e\rangle\) satisfy
    \(|\dot r|\leq|a|\) and \(\dot s=r\langle b,e\rangle\). Hence
    \((r,s)\) is horizontal in the classical Grushin plane and
    \begin{equation*}
        \ell(r,s)
        \leq
        \int_0^t\left(\dot r^2+\langle b,e\rangle^2\right)^{1/2}\,d\tau
        \leq \ell(\eta).
    \end{equation*}
    The projection of \(\gamma\) is the classical Grushin geodesic from
    \((0,0)\) given by
    \begin{equation*}
        \bar\gamma(t)
        =
        \left(
        \frac{|u_0|}{\omega}\sin(\omega t),
        \omega g(t;u_0,\omega)
        \right),
        \qquad 0\leq t\leq\frac{\pi}{\omega}.
    \end{equation*}
    By the singular-point optimal synthesis
    \cite[Section~1.6.3]{juilletphdthesis} (see also
    \cite[Lemma~17]{borza2022}), \(\bar\gamma|_{[0,t]}\) is uniquely
    minimizing. Therefore,
    \begin{equation*}
        \ell(\eta)
        \geq \ell(r,s)
        \geq \ell(\bar\gamma|_{[0,t]})
        =\ell(\gamma|_{[0,t]}).
    \end{equation*}
    Equality between $\ell(\eta)$ and $\ell(\gamma\rvert_{[0,t]})$ forces equality with $\ell(r,s)$, wherein the integral bounding $\ell(r,s)$ is then equal to $\ell(\eta)$, so that the integrands must be equal. The equality condition in the Cauchy--Schwarz inequality then shows that \(a\) must be parallel to \(x\) and \(b\) must be
    parallel to \(e\); the endpoints then fix their directions, so
    \(\eta=\gamma|_{[0,t]}\). Thus no cut occurs before \(\pi/\omega\).
    Conversely,
    \((r,s)\mapsto(rv,y_0+se)\) is an isometric embedding of the classical
    Grushin plane into \(\mathbb G^{n+m}\), so the classical synthesis also
    excludes minimality beyond \(\pi/\omega\). Hence
    \(t_{\operatorname{cut}}(\gamma)=\pi/\omega\).

    At this time, \cref{radial x formula,g formula} give
    \begin{equation*}
        \gamma\left(\frac{\pi}{\omega}\right)
        =
        \left(
        0,
        y_0+\frac{\pi|u_0|^2}{2\omega^3}w_0
        \right).
    \end{equation*}
    Moreover,
    \begin{equation*}
        \left\{
        \frac{\pi|u|^2}{2|w|^3}w:
        u\neq0,\ w\neq0
        \right\}
        =\mathbb R^m\setminus\{0\}.
    \end{equation*}
    Therefore every \((0,y)\), \(y\neq y_0\), is a cut point, and
    \(\operatorname{Cut}(0,y_0)=\{(0,y):y\neq y_0\}\).
\end{proof}
Let \(q_0=(x_0,y_0)\in\mathbb{G}^{n+m}\) and let
\(q_1=(x_1,y_1)\notin\operatorname{Cut}(q_0)\). By
\cref{radial cut time theorem}, there is a unique time-one
minimizing geodesic \(\gamma(\cdot;u_0,w_0)\) connecting \(q_0\) to
\(q_1\). Its distortion coefficient is
\begin{align*}
    \beta_t(q_0,q_1)
    =
    \frac{
        J_{(x_0,y_0)}(t;u_0,w_0)
    }{
        J_{(x_0,y_0)}(1;u_0,w_0)
    }.
\end{align*}
Writing \(\omega=|w_0|\), \cref{Conjugacy Profile} gives
\begin{align}\label{distortion formula}
    \beta_t(q_0,q_1)
    =
    \begin{cases}
        \displaystyle
        \left(
        \frac{\sin(\omega t)}{\sin\omega}
        \right)^{n-1}
        \left(
        \frac{
            G(\omega t,\omega x_0,u_0)
        }{
            G(\omega,\omega x_0,u_0)
        }
        \right)^{m-1}
        \left(
        t\frac{
             H(\omega t,\omega x_0,u_0)
         }{
             H(\omega,\omega x_0,u_0)
         }
        \right),
         & \omega\neq0, \\[10pt]
        \displaystyle
        t^n
        \left(
        \frac{
            t|x_0|^2+t^2\langle x_0,u_0\rangle
            +\frac{t^3}{3}|u_0|^2
        }{
            |x_0|^2+\langle x_0,u_0\rangle
            +\frac13|u_0|^2
        }
        \right)^m,
         & \omega=0,
    \end{cases}
\end{align}
where the condition \(\omega=0\) is equivalent to \(y_1=y_0\), and where
\begin{align}\label{H formula}
    H(z,\zeta,u)
    :=
    \sin z\,|\zeta|^2
    +z\sin z\,\langle\zeta,u\rangle
    +(\sin z-z\cos z)|u|^2
\end{align}
and
\begin{align*}
    G(z,\zeta,u)
    :=
    \frac{z+\sin z\cos z}{2}|\zeta|^2
    +\sin^2z\,\langle\zeta,u\rangle
    +\frac{z-\sin z\cos z}{2}|u|^2.
\end{align*}
Formula~\cref{H formula} follows from the uncompleted-square form of
\cref{P simplification}.

We now prove \cref{Radial Grushin MCP Theorem}. By the observations
following \cref{MCP theorem Rizzi}, unboundedness rules out \(K>0\),
while monotonicity and the metric dilations reduce the remaining cases
to determining the smallest \(N\) for which \(\operatorname{MCP}(0,N)\)
holds. By \cref{MCP theorem Rizzi}, this is equivalent to proving
\(\beta_t(q_0,q_1)\geq t^N\) for every
\(q_0\in\mathbb G^{n+m}\), every
\(q_1\notin\operatorname{Cut}(q_0)\), and every \(t\in[0,1]\). We use
\cite[Proposition~62]{BarilariRizzi2019} for the classical Grushin
factor and logarithmic differentiation for the remaining estimate.

\begin{proof}[Proof of \cref{Radial Grushin MCP Theorem}]
    We first prove that $\beta_t(q_0,q_1)\geq t^{n+4m}$. Suppose first that \(\omega\neq0\). Since the time-one endpoint lies
    before the cut time, \cref{radial cut time theorem} gives
    \(0<\omega<\pi\). By \cref{distortion formula}, it is enough to prove
    \begin{align}\label{radial distortion estimates}
        \frac{
            G(\omega t,\omega x_0,u_0)
        }{
            G(\omega,\omega x_0,u_0)
        }
        \geq t^4, \qquad
        t\frac{
             H(\omega t,\omega x_0,u_0)
         }{
             H(\omega,\omega x_0,u_0)
         }
        \geq t^5, \quad \text{ and } \quad
        \frac{\sin(\omega t)}{\sin\omega}\geq t.
    \end{align}
    Indeed, inserting \cref{radial distortion estimates} into
    \cref{distortion formula} gives the exponent
    \((n-1)+4(m-1)+5=n+4m\). The last inequality in
    \cref{radial distortion estimates} follows from the concavity of
    \(\sin\) on \([0,\pi]\).

    We next prove the middle inequality in
    \cref{radial distortion estimates}. Put \(\zeta=\omega x_0\); then
    \begin{equation*}
        \frac{H(\omega,\zeta,u_0)}{\sin\omega}
        =
        \left|\zeta+\frac{\omega}{2}u_0\right|^2
        +
        \left(
        1-\omega\cot\omega-\frac{\omega^2}{4}
        \right)|u_0|^2,
    \end{equation*}
    and \cref{technical} shows that \(H(\omega,\zeta,u_0)>0\) whenever
    \((\zeta,u_0)\neq(0,0)\). For \(a,b\in\mathbb{R}\), consider the
    classical Grushin plane with base point \(q_0=(a/\omega,y_0)\) and
    initial covector \(b\,dx+\omega\,dy\). By
    \cite[Propositions~61 and~62]{BarilariRizzi2019}, its distortion
    coefficient is \(tH(t\omega,a,b)/H(\omega,a,b)\), and hence
    \begin{equation}\label{classical H estimate}
        H(t\omega,a,b)\geq t^4H(\omega,a,b),
        \qquad
        0\leq t\leq1.
    \end{equation}
    The same formula and estimate also follow from
    \cref{non-straight line distortion,global inequality} by specializing to
    \(\alpha=1\). If \(u_0\neq0\), set \(e=u_0/|u_0|\),
    \(a=\langle\zeta,e\rangle\), and
    write \(\zeta=a e+\zeta_\perp\), where
    \(\zeta_\perp\perp e\). Then
    \begin{equation*}
        H(\omega,\zeta,u_0)
        =
        H(\omega,a,|u_0|)+\sin\omega\,|\zeta_\perp|^2.
    \end{equation*}
    Combining \cref{classical H estimate} with
    \(\sin(t\omega)\geq t\sin\omega\geq t^4\sin\omega\) gives
    \begin{align*}
        H(t\omega,\zeta,u_0)
         & \geq
        t^4H(\omega,a,|u_0|)
        +t^4\sin\omega\,|\zeta_\perp|^2
        =t^4H(\omega,\zeta,u_0).
    \end{align*}
    When \(u_0=0\), the same conclusion follows directly from
    \(H(\omega,\zeta,0)=\sin\omega\,|\zeta|^2\).

    It remains to prove the first inequality in
    \cref{radial distortion estimates}. For \(0<z<\pi\),
    \begin{equation*}
        G(z,\zeta,u_0)
        =
        \int_0^z
        |\zeta\cos s+u_0\sin s|^2\,ds,
    \end{equation*}
    so \(G(z,\zeta,u_0)>0\) whenever
    \((\zeta,u_0)\neq(0,0)\). Logarithmic differentiation therefore
    reduces the first inequality in \cref{radial distortion estimates} to
    \(z\partial_zG(z,\zeta,u_0)\leq4G(z,\zeta,u_0)\). Direct computation
    gives $4G-z\partial_zG
        =
        A(z)|\zeta|^2
        +C(z)\langle\zeta,u_0\rangle
        +B(z)|u_0|^2$,
    where
    \begin{equation*}
        \begin{gathered}
            A(z)=z(1+\sin^2z)+2\sin z\cos z,
            \qquad
            B(z)=z(1+\cos^2z)-2\sin z\cos z, \\
            C(z)=2\sin z(2\sin z-z\cos z).
        \end{gathered}
    \end{equation*}
    The coefficient \(A\) is positive on \((0,\pi)\): this is immediate
    for \(z\leq\pi/2\), while for \(z\geq\pi/2\) one has
    \(A(z)\geq z-1>0\). The positivity of
    \(B\) follows directly from \cref{second comparison}.
    Furthermore, $4A(z)B(z)-C(z)^2
        =
        8\left(
        z^2+z\sin z\cos z-2\sin^2z
        \right)$. Set \(F(z)=z^2+z\sin z\cos z-2\sin^2z\). On
    \((0,\pi/2)\), one has \(F^{(4)}(z)=8z\sin(2z)>0\), while
    \(F^{(j)}(0)=0\) for \(0\leq j\leq3\). Hence \(F(z)>0\) on this
    interval. On \([\pi/2,\pi)\),
    \begin{align*}
        F''(z)
        =
        2-2\cos(2z)-2z\sin(2z)>0,
    \end{align*}
    and \(F'(\pi/2)=\pi/2>0\). Since
    \(F(\pi/2)=\pi^2/4-2>0\), the function \(F\) remains positive on
    \([\pi/2,\pi)\). Since \(A(z)>0\), \(B(z)\geq0\), and
    \(4A(z)B(z)-C(z)^2>0\), we obtain
    \begin{equation*}
        4G(z,\zeta,u_0)-z\partial_zG(z,\zeta,u_0)\geq0.
    \end{equation*}
    This proves the first inequality.

    We have therefore proved all three inequalities in
    \cref{radial distortion estimates}, and hence
    \(\beta_t(q_0,q_1)\geq t^{n+4m}\) when \(\omega\neq0\). The
    \(\omega=0\) branch of
    \cref{distortion formula} is the limit of the \(\omega\neq0\) branch
    as \(\omega\to0\); hence the same inequality holds when \(\omega=0\).

    To prove sharpness, choose \(x_0\neq0\), \(\omega=0\), and
    \(u_0=-3x_0\). The corresponding endpoint
    \((-2x_0,y_0)\) is regular by \cref{radial cut locus formula}, and
    \begin{align*}
        \beta_t(q_0,q_1)
        =
        t^n
        \left(
        t-3t^2+3t^3
        \right)^m.
    \end{align*}
    The right-hand side equals \(1\) at \(t=1\) and has derivative
    \(n+4m\) there. If \(\beta_t(q_0,q_1)\geq t^N\) held on
    \([0,1]\), then
    \(f(t):=\beta_t(q_0,q_1)-t^N\) would satisfy
    \(f(t)\geq f(1)=0\). Hence
    \((f(t)-f(1))/(t-1)\leq0\) for \(t<1\). Letting \(t\uparrow1\)
    gives \(n+4m-N=f'_-(1)\leq0\), and therefore
    \(N\geq n+4m\).
\end{proof}

\section{\texorpdfstring{The measure contraction property of the
      $\alpha$-Grushin plane}{The measure contraction property of the alpha-Grushin plane}}
\label{sec:alpha-grushin-mcp}

\subsection{\texorpdfstring{Distortion coefficients of the
        $\alpha$-Grushin plane}{Distortion coefficients of the alpha-Grushin plane}}

In this section, we resolve the conjecture of \cite{borza2022}
concerning the curvature exponent of the \(\alpha\)-Grushin plane. Let
\(\alpha\in\mathbb R\) with \(\alpha\geq1\), and consider on \(\mathbb R^2\)
the vector fields
\begin{equation*}
    X=\partial_x,
    \qquad
    Y_\alpha=|x|^\alpha\partial_y.
\end{equation*}
We denote by \(\mathbb G_\alpha^2\) the resulting metric measure space,
equipped with the Carnot--Carath\'eodory distance and two-dimensional
Lebesgue measure. The structure is Riemannian on \(\{x\neq0\}\). When
\(\alpha\in\mathbb N\), replacing \(Y_\alpha\) by
\(x^\alpha\partial_y\) gives a smooth Hörmander family without changing
the horizontal curves, their lengths, or the induced
distance, and the step at each point of \(\{x=0\}\) is \(\alpha+1\). For noninteger
\(\alpha>1\), the original field \(Y_\alpha\) is \(C^1\), hence locally
Lipschitz, so the controlled differential equation defining horizontal
curves is well posed. Any two points can be joined by a horizontal curve and the
Carnot--Carath\'eodory distance is finite. It is moreover complete, and
therefore geodesic. Finally, there are no nonconstant abnormal minimizers. Thus, when \(\alpha\) is not
an integer, \(\mathbb G_\alpha^2\) is not a smooth sub-Riemannian manifold
in the classical Hörmander sense, but all the sub-Riemannian constructions
used in this paper remain valid.

Let \(\sin_\alpha\) be the solution of
\(\ddot y=-\alpha |y|^{2\alpha-2}y\), \(y(0)=0\), and \(\dot y(0)=1\),
and define \(\cos_\alpha=(\sin_\alpha)'\). We will repeatedly use
the energy identity
\begin{equation}\label{alpha energy identity}
    \cos_\alpha(t)^2+|\sin_\alpha(t)|^{2\alpha}=1.
\end{equation}
Let \(\pi_\alpha\) denote the first positive zero of
\(\sin_\alpha\). Both \(\sin_\alpha\) and \(\cos_\alpha\) are
\(2\pi_\alpha\)-periodic and satisfy
\begin{equation}\label{alpha antiperiodicity}
    \sin_\alpha(t+\pi_\alpha)=-\sin_\alpha(t),
    \qquad
    \cos_\alpha(t+\pi_\alpha)=-\cos_\alpha(t).
\end{equation}
For later use put $\tan_\alpha=\sin_\alpha/\cos_\alpha$ wherever it is defined, and $\cot_\alpha=(\tan_\alpha)^{-1}$. We recall from \cite{borza2022} the normal trajectories obtained from \cref{eq:hamiltonequation} and the optimal synthesis
of $\mathbb G^2_\alpha$.

\begin{theorem}
    Let \(q_0=(x_0,y_0)\in\mathbb G_\alpha^2\), and let
    \(\gamma(t)=(x(t),y(t))\) be a nonconstant, non-straight-line normal
    geodesic with initial covector \(p_0=(u_0,v_0)\) and constant speed $\kappa
            =
            (u_0^2+|x_0|^{2\alpha}v_0^2)^{1/2}$.
    There exist \(A,\omega\in\mathbb R\) and
    \(\phi\in[0,2\pi_\alpha)\), with \(A\omega=\kappa>0\), such that
    \begin{equation*}
        \begin{aligned}
            x(t)
             & =
            A\sin_\alpha(\omega t+\phi), \\
            y(t)
             & =
            y_0+
            v_0\frac{|A|^{2\alpha}}{\omega(\alpha+1)}
            \left(
            \omega t
            -\sin_\alpha(\omega t+\phi)\cos_\alpha(\omega t+\phi)
            +\sin_\alpha(\phi)\cos_\alpha(\phi)
            \right).
        \end{aligned}
    \end{equation*}
    When \(q_0=(0,y_0)\) is singular, the phase is
    \(\phi=0\) if \(u_0=\kappa\), and \(\phi=\pi_\alpha\) if
    \(u_0=-\kappa\). Moreover,
    \begin{equation*}
        \begin{aligned}
            t_{\operatorname{cut}}(\gamma)
            =
            \frac{\pi_\alpha}{|\omega|}, \quad \text{ and } \quad \operatorname{Cut}(q_0)
            =
            \begin{cases}
                \{0\}\times\bigl(\mathbb R\setminus\{y_0\}\bigr),
                 & x_0=0,    \\[4pt]
                \displaystyle
                \left\{
                (-x_0,y):
                |y-y_0|
                \geq
                \frac{\pi_\alpha|x_0|^{\alpha+1}}{\alpha+1}
                \right\},
                 & x_0\neq0.
            \end{cases}
        \end{aligned}
    \end{equation*}
    The horizontal straight lines $(x_0\pm \kappa t,y_0)$ are geodesics and optimal for all time. They correspond to $\omega=0$.
\end{theorem}
For a Riemannian initial point $q_0$, the relation
\(x_0=A\sin_\alpha\phi\) determines \(A\) from \(\phi\), so the exponential map at such a point may be parametrized by \((\omega,\phi)\) only. The Jacobian of the exponential map evaluated at time \(t\) can be written as
\begin{equation*}
    \det D_{(\omega,\phi)}
    \mathrm{exp}_{q_0}^t
    =
    t \,
    \frac{x_0}{\sin_\alpha(\phi)^2}
    \left|
    \frac{x_0}{\sin_\alpha\phi}
    \right|^{\alpha+1} G(t\omega,\phi),
\end{equation*}
where $G(z,\phi)
    :=
    \sin_\alpha(z+\phi)\cos_\alpha\phi
    -
    \cos_\alpha(z+\phi)
    (z\cos_\alpha\phi+\sin_\alpha\phi)$.
Since there
are no conjugate times before the cut time
\(\pi_\alpha/|\omega|\), one has $G(z,\phi)\neq0$ for $0<|z|<\pi_\alpha$ and $\phi\in[0,2\pi_\alpha)$.

Suppose that a point \(q_1\in\mathbb G_\alpha^2\) does not lie on the same horizontal line as \(q_0\).
Then, if \(q_1\notin\operatorname{Cut}(q_0)\), we have
\begin{equation}\label{non-straight line distortion}
    \beta_t(q_0,q_1)
    =
    \frac{
        \det D_{(\omega,\phi)}
        \mathrm{exp}_{q_0}^t
    }{
        \det D_{(\omega,\phi)}
        \mathrm{exp}_{q_0}
    }
    =
    t\frac{G(t\omega,\phi)}{G(\omega,\phi)}, \qquad \text{ for every \(t\in[0,1]\), }
\end{equation}
where \((\omega,\phi)\) are the unique
parameters of the minimizing geodesic satisfying
\(q_1=\mathrm{exp}_{q_0}(\omega,\phi)\).
The distortion formula \cref{non-straight line distortion} extends to the case where $q_0$ is singular by taking \(\phi=0\). Finally, if \(q_0\) and \(q_1\) do lie on the same
horizontal line, which corresponds to the case where \(v_0=0\) (or $\omega = 0$), then the unique minimizing geodesic joining them is the horizontal segment $\gamma(t)=(x_0+tu_0,y_0)$, where $u_0$ is the unique parameter such that \(q_1=\mathrm{exp}_{q_0}(u_0,0)\), and
\begin{equation*}
    \beta_t(q_0,q_1)
    =
    t
    \frac{
        |x_0+tu_0|^{2\alpha}(x_0+tu_0)-|x_0|^{2\alpha}x_0
    }{
        |x_0+u_0|^{2\alpha}(x_0+u_0)-|x_0|^{2\alpha}x_0
    },
\end{equation*}
which is the limit of \cref{non-straight line distortion} as $\omega \to 0$.

As in the proof of
\cref{Radial Grushin MCP Theorem} and via \cref{lem:jacobian-mcp}, the
\(\operatorname{MCP}(0,N)\) problem reduces to proving
\(\beta_t(q_0,q_1)\geq t^N\) for every \(t\in[0,1]\), separately for
the $\omega=0$ and $\omega\neq 0$ geodesics.

Following \cite{borza2022}, let \(m_\alpha\in[-3,-2]\) be the unique
solution of
\begin{equation}
    \label{eq:malpha-definition}
    (m+1)|m+1|^{2\alpha}
    -
    \bigl((2\alpha+1)m+1\bigr)
    =
    0,
\end{equation}
and define $M_\alpha:=|m_\alpha+1|^{2\alpha}$ and $N_\alpha:=1+M_\alpha$.
In \cite{borza2022}, it was conjectured that \(N_\alpha\) is the
curvature exponent of the \(\alpha\)-Grushin plane, and the
bound \(\beta_t(q_0,q_1)\geq t^{N_\alpha}\) was established when either \(q_0\) is
singular or \(q_0\) and \(q_1\) lie on the same horizontal line. It therefore only remains to establish
this bound for non-horizontal geodesics originating from a Riemannian
point, whose distortion coefficient is given by
\cref{non-straight line distortion}.

We define
\begin{equation*}
    g(z,\phi)
    :=
    z\frac{\partial_zG(z,\phi)}{G(z,\phi)},
    \qquad
    0<|z|<\pi_\alpha.
\end{equation*}
Since \(G(\cdot,\phi)\) is \(C^2\) and does not vanish for
\(0<|z|<\pi_\alpha\), for each fixed \(\phi\) the maps
\(z\mapsto\log|G(z,\phi)|\) and \(z\mapsto g(z,\phi)\) are respectively
\(C^2\) and \(C^1\) on each half-interval. Moreover,
\begin{equation*}
    \partial_z\log|G(z,\phi)|
    =
    \frac{G_z(z,\phi)}{G(z,\phi)}
    =
    \frac{g(z,\phi)}{z}.
\end{equation*}
To obtain \cref{alpha Grushin MCP Theorem} it is enough to prove
\begin{equation}\label{globalinequality}
    g(z,\phi)\leq M_\alpha,
    \qquad
    0<|z|<\pi_\alpha,\quad
    \phi\in[0,2\pi_\alpha).
\end{equation}
Indeed, for \(\omega\neq0\) and \(0<t\leq1\), the nonvanishing of
\(G\) gives
\begin{equation*}
    \log
    \frac{|G(\omega,\phi)|}{|G(t\omega,\phi)|}
    =
    \int_t^1
    \frac{g(s\omega,\phi)}{s}\,ds
    \leq
    M_\alpha\log\frac1t.
\end{equation*}
Therefore
\(G(t\omega,\phi)/G(\omega,\phi)\geq t^{M_\alpha}\), and
\cref{non-straight line distortion} gives
\(\beta_t(q_0,q_1)\geq t^{1+M_\alpha}=t^{N_\alpha}\). Finally, $G(t\omega,\phi)/G(\omega,\phi)\geq t^{M_\alpha}$ is implied by \cref{globalinequality}.
Therefore, the rest of this work will be devoted to proving the following theorem.
\begin{theorem}\label{global inequality}
    For all $(z,\phi)$ with $0<|z|<\pi_\alpha$ and $\phi\in[0,2\pi_\alpha)$, one has
    \begin{equation}
        \label{eq:MCPtoprove}
        g(z,\phi)\leq M_\alpha.
    \end{equation}
    Moreover, the supremum of $g$ over $0<|z|<\pi_\alpha$ and $\phi\in [0,2\pi_\alpha)$ is exactly $M_\alpha$.
\end{theorem}

\subsection{Equivalent forms of the conjectured curvature exponent}
\label{subsec:equivalent-curvature-exponent}

We record several equivalent expressions for \(M_\alpha\) that will
be used in the proof of \cref{global inequality} in the next section. Define
\begin{equation*}
    p_\alpha(t)
    :=
    \begin{cases}
        \displaystyle
        \frac{t|t|^{2\alpha}-1}{t-1},
         & t\neq1, \\[6pt]
        2\alpha+1,
         & t=1.
    \end{cases}
\end{equation*}
When \(\alpha\in\mathbb N\), this reduces to
\(p_\alpha(t)=\sum_{k=0}^{2\alpha}t^k\).

\begin{proposition}[Equivalent formulas for \(M_\alpha\)]
    For every real \(\alpha\geq1\), the function \(p_\alpha\) has a unique
    minimizer \(t_\alpha\in(-1,-1/2]\), and
    \begin{equation}
        \label{Malpha equivalent forms}
        M_\alpha
        =
        \max_{L>1}
        \frac{(2\alpha+1)L}
        {(L-1)^{2\alpha+1}+1} =
        \frac{2\alpha+1}
        {\displaystyle\min_{t\in\mathbb R}p_\alpha(t)}
        \geq
        \frac{4}{3}(2\alpha+1).
    \end{equation}
    The maximum over \(L>1\) in \cref{Malpha equivalent forms} is attained uniquely at
    \begin{equation}\label{eq:Lalpha-range}
        L_\alpha
        =
        1-t_\alpha
        =
        \frac{m_\alpha}{m_\alpha+1}
        =
        \frac{M_\alpha-1}{2\alpha}
        \in[3/2,2).
    \end{equation}
\end{proposition}

\begin{proof}
    For \(t\geq0\), one has \(p_\alpha(t)\geq1\). For \(t<0\), we find that
    \begin{equation*}
        p_\alpha(t) =\frac{1+|t|^{2\alpha+1}}{1+|t|}, \quad \text{ and } \quad
        p_\alpha'(t)
        =\frac{H_\alpha(t)}{(1+|t|)^2},
    \end{equation*}
    where $H_\alpha(t) :=1-(2\alpha+1)|t|^{2\alpha}
        +2\alpha t|t|^{2\alpha}$. Moreover, $H_\alpha'(t) =2\alpha(2\alpha+1)|t|^{2\alpha-1}(1+|t|)>0$,
    \(H_\alpha(-1)=-4\alpha<0\) and
    \(H_\alpha(-1/2)=1-(3\alpha+1)/4^{\alpha}\geq0\).
    Thus \(H_\alpha\) has a unique zero
    \(t_\alpha\in(-1,-1/2]\). Since \(p_\alpha(t)\to+\infty\) as
    \(t\to-\infty\) and
    \(p_\alpha(t_\alpha)\leq p_\alpha(-1/2)\leq3/4\), this is the
    unique global minimizer of \(p_\alpha\).

    The bijection $[-3,-2]\to[-1,-1/2]:
        m \mapsto t := 1/(m+1)$
    transforms the defining equation for \(m_\alpha\) into
    \(H_\alpha(t)=0\). Hence
    \(t_\alpha=1/(m_\alpha+1)\). Using
    \(H_\alpha(t_\alpha)=0\), we obtain
    \begin{equation*}
        p_\alpha(t_\alpha)
        =(2\alpha+1)|t_\alpha|^{2\alpha},
        \qquad
        M_\alpha
        =|t_\alpha|^{-2\alpha}
        =2\alpha+1-2\alpha t_\alpha
        =\frac{2\alpha+1}{p_\alpha(t_\alpha)}.
    \end{equation*}

    Finally, the bijection $(1,\infty)\to(-\infty,0):
        L\mapsto t:=1-L$ gives
    \begin{equation*}
        \frac{(2\alpha+1)L}{(L-1)^{2\alpha+1}+1}
        =\frac{2\alpha+1}{p_\alpha(t)}.
    \end{equation*}
    Therefore the maximum is attained uniquely at
    \(L_\alpha=1-t_\alpha=m_\alpha/(m_\alpha+1)\), and the preceding
    formula for \(M_\alpha\) gives
    \(L_\alpha=(M_\alpha-1)/(2\alpha)\). Evaluating the maximum at
    \(L=3/2\) yields
    \begin{equation*}
        M_\alpha
        \geq
        \frac{3(2\alpha+1)}
        {2\bigl(1+2^{-(2\alpha+1)}\bigr)}
        \geq
        \frac{4}{3}(2\alpha+1),
    \end{equation*}
    as required.
\end{proof}

\begin{remark}[Irrationality for integer \(\alpha\)]
    \label{rem:alpha-curvature-irrationality}
    For every integer \(\alpha\geq2\), the exponent \(N_\alpha\) is irrational.
    As shown in the preceding proof, \(H_\alpha\) is strictly increasing on
    \((-\infty,0)\) and vanishes at \(t_\alpha\).
    Since \(H_\alpha(-1)<0\) and
    \(H_\alpha(-1/2)=1-(3\alpha+1)/4^\alpha>0\), we have
    \(-1<t_\alpha<-1/2\). The relation \(m_\alpha=1/t_\alpha-1\)
    therefore gives \(-3<m_\alpha<-2\).
    The defining equation \eqref{eq:malpha-definition} for \(m_\alpha\) becomes
    \begin{equation}
        \label{eq:malpha-polynomial}
        (m_\alpha+1)^{2\alpha+1}-(2\alpha+1)m_\alpha-1=0.
    \end{equation}
    This is a monic polynomial with integer coefficients, so the rational
    root theorem implies that any rational root is an integer. In particular
    \(m_\alpha\) is irrational, and dividing \cref{eq:malpha-polynomial} by \(m_\alpha+1\) gives
    \begin{equation*}
        N_\alpha
            =1+(m_\alpha+1)^{2\alpha}
        =2\alpha+2-\frac{2\alpha}{m_\alpha+1},
    \end{equation*}
    which implies that \(N_\alpha\) is irrational as well.
    For $\alpha=1$, we recover $N_1=5$.
\end{remark}

\begin{remark}[Large-\(\alpha\) asymptotics]
    \label{rem:alpha-curvature-asymptotic}
    As \(\alpha\to\infty\), it holds that
    \begin{equation}
        \label{eq:assymptoticsNalpha}
        N_\alpha
        =
        4\alpha+2-\log(4\alpha+2)+o(1).
    \end{equation}
    Indeed, since \(L_\alpha\) is the maximizer in
    \cref{Malpha equivalent forms}, differentiating gives
    \begin{equation*}
        (L_\alpha-1)^{2\alpha}(2\alpha L_\alpha+1)=1,
        \qquad
        N_\alpha=2\alpha L_\alpha+2.
    \end{equation*}
    By \cref{eq:Lalpha-range}, taking logarithms in the first identity
    shows that \(L_\alpha\to2\) and
    \(2-L_\alpha=O(\log\alpha/\alpha)\). Therefore, we have
    \begin{equation*}
        2\alpha(2-L_\alpha)
        =
        -2\alpha\log(L_\alpha-1)+o(1)
        =
        \log(2\alpha L_\alpha+1)+o(1)
        =
        \log(4\alpha+2)+o(1).
    \end{equation*}
    Substitution into \(N_\alpha=4\alpha+2-2\alpha(2-L_\alpha)\)
    gives \cref{eq:assymptoticsNalpha}.
\end{remark}

We next study \(g(z,\phi)\) as \((z,\phi)\) approaches
\((0,j\pi_\alpha)\) and show that its largest possible limiting value
is precisely \(M_\alpha\).
\begin{proposition}\label{thm:blowup}
    It holds that
    \begin{equation*}
        \limsup_{\substack{(z,\phi)\to(0,0)\\ z\neq0}}
        g(z,\phi)
        =
        \max_{(u,v)\in\mathbb S^1}\Phi(u,v)
        =
        \frac{2\alpha+1}
        {\displaystyle\min_{t\in\mathbb R}p_\alpha(t)}
        =
        M_\alpha,
    \end{equation*}
    where
    \begin{equation*}
        \Phi(u,v)
        :=
        \frac{|u+v|^{2\alpha}}
        {\displaystyle\int_0^1|v+\tau u|^{2\alpha}\,d\tau},
        \qquad
        (u,v)\in\mathbb R^2\setminus\{(0,0)\}.
    \end{equation*}
\end{proposition}

\begin{proof}
    The expansions
    \begin{equation*}
        \sin_\alpha(t)
        =
        t+O(|t|^{2\alpha+1}),
        \qquad
        \cos_\alpha(t)
        =
        1+O(|t|^{2\alpha}), \qquad \text{ as } t\to0,
    \end{equation*}
    and the bound
    \begin{equation*}
        \bigl||a|^{2\alpha-2}a-|b|^{2\alpha-2}b\bigr|
        \leq
        (2\alpha-1)(|a|+|b|)^{2\alpha-2}|a-b|
    \end{equation*}
    yield
    \begin{align*}
        z\cos_\alpha\phi+\sin_\alpha\phi
         & =z+\phi
        +O\bigl(\lVert(z,\phi)\rVert^{2\alpha+1}\bigr), \\
        |\sin_\alpha(z+\phi)|^{2\alpha-2}
        \sin_\alpha(z+\phi)
         & =|z+\phi|^{2\alpha-2}(z+\phi)
        +O\bigl(\lVert(z,\phi)\rVert^{4\alpha-1}\bigr).
    \end{align*}
    Since direct differentiation of $G$ gives
    \begin{equation*}
        \partial_zG(z,\phi)
        =
        \alpha|\sin_\alpha(z+\phi)|^{2\alpha-2}
        \sin_\alpha(z+\phi)
        \bigl(
        z\cos_\alpha\phi+\sin_\alpha\phi
        \bigr),
    \end{equation*}
    it follows uniformly as
    \((z,\phi)\to(0,0)\) that
    \begin{equation*}
        \partial_zG(z,\phi)
        =
        \alpha|z+\phi|^{2\alpha}
        +O\bigl(\lVert(z,\phi)\rVert^{4\alpha}\bigr).
    \end{equation*}

    Noting that \(G(0,\phi)=0\), we also find that
    \begin{equation*}
        G(z,\phi)
        =
        z\int_0^1\partial_zG(\tau z,\phi)\,d\tau =\alpha zI(z,\phi)
        +O\bigl(|z|\lVert(z,\phi)\rVert^{4\alpha}\bigr),
    \end{equation*}
    where $I(z,\phi)
        :=
        \int_0^1|\phi+\tau z|^{2\alpha}\,d\tau$. The function \(I\) is positive away from the origin, continuous,
    and homogeneous of degree \(2\alpha\). Compactness of
    \(\mathbb S^1\) gives a constant \(c_\alpha>0\) such that
    \begin{equation*}
        I(z,\phi)
        \geq
        c_\alpha\lVert(z,\phi)\rVert^{2\alpha}.
    \end{equation*}
    Therefore, for \(z\neq0\),
    \begin{equation}\label{g Phi asymptotic}
        g(z,\phi)
        =
        \frac{
            |z+\phi|^{2\alpha}
            +O\bigl(\lVert(z,\phi)\rVert^{4\alpha}\bigr)
        }{
            I(z,\phi)
            +O\bigl(\lVert(z,\phi)\rVert^{4\alpha}\bigr)
        }
        =
        \Phi(z,\phi)
        +O\bigl(\lVert(z,\phi)\rVert^{2\alpha}\bigr).
    \end{equation}
    The function \(\Phi\) is continuous away from the origin and
    homogeneous of degree zero. Since the directions with \(u\neq0\)
    are dense in \(\mathbb S^1\), it follows that
    \begin{equation*}
        \limsup_{\substack{(z,\phi)\to(0,0)\\z\neq0}}
        g(z,\phi)
        =
        \max_{(u,v)\in\mathbb S^1}\Phi(u,v).
    \end{equation*}
    It remains to compute this maximum. For \(u+v\neq0\), set
    \(t=v/(u+v)\). Then
    \begin{equation*}
        I(u,v)
        =
        |u+v|^{2\alpha}
        \int_0^1|t+\tau(1-t)|^{2\alpha}\,d\tau
        =
        \frac{|u+v|^{2\alpha}}{2\alpha+1}p_\alpha(t).
    \end{equation*}
    For \(t\neq1\), the last equality follows by integrating
    \(|r|^{2\alpha}\). If \(t=1\), then \(u=0\), and it follows from
    \(I(0,v)=|v|^{2\alpha}\) and \(p_\alpha(1)=2\alpha+1\).
    Therefore $\Phi(u,v)
        =
        (2\alpha+1)/p_\alpha(t)$.
    When \(u+v=0\), one has \(\Phi(u,v)=0\), so this direction does not
    contribute to the maximum. Conversely, every \(t\in\mathbb R\)
    is obtained from a direction proportional to \((1-t,t)\), and we conclude with
    \cref{Malpha equivalent forms}.
\end{proof}

\subsection{\texorpdfstring{The curvature exponent of the
        $\alpha$-Grushin plane}{The curvature exponent of the alpha-Grushin plane}}

We begin by determining the behavior of \(g(z,\phi)\) as \(z\to0\)
with \(\phi\) fixed.

\begin{lemma}\label{lem:boundary-input}
    For every fixed \(\phi\in\mathbb R\),
    \begin{equation*}
        \lim_{z\to0}g(z,\phi)
        =
        \begin{cases}
            1,         & \sin_\alpha\phi\neq0,       \\
            2\alpha+1, & \phi\in\pi_\alpha\mathbb Z.
        \end{cases}
    \end{equation*}
\end{lemma}

\begin{proof}
    If \(\sin_\alpha\phi\neq0\), then \(G(0,\phi)=0\) and $\partial_zG(0,\phi)
        =
        \alpha|\sin_\alpha(\phi)|^{2\alpha}\neq0$. A Taylor expansion in the \(z\)-variable gives
    \begin{equation*}
        g(z,\phi)
        =
        z\frac{\partial_zG(z,\phi)}{G(z,\phi)}
        =
        \frac{
            z\,\partial_zG(0,\phi)+O(z^2)
        }{
            z\,\partial_zG(0,\phi)+O(z^2)
        }
        \longrightarrow 1
    \end{equation*}
    as \(z\to0\). If \(\phi=j\pi_\alpha\) for some \(j\in\mathbb Z\),
    the sign-change identities for \(\sin_\alpha\) and \(\cos_\alpha\)
    imply \(g(z,\phi)=g(z,0)\). Since \(\Phi(z,0)=2\alpha+1\),
    \cref{g Phi asymptotic} gives \(g(z,0)\to2\alpha+1\).
\end{proof}

In the sequel, we will make the following substitutions to ease the notation:
\begin{equation*}
    S=\sin_\alpha(z+\phi),\quad C=\cos_\alpha(z+\phi),
    \quad A=\sin_\alpha\phi,\quad B=\cos_\alpha\phi,
    \quad Q=A+zB.
\end{equation*}
We now show a Riccati-type equation for $g$.
\begin{lemma}[Riccati equation]
    On the set where $G\neq0$, $S\neq0$, and $Q\neq0$, the function $g$ satisfies
    \begin{equation}\label{Riccati Equation}
        \partial_z g
        =
        \alpha z|S|^{2\alpha-2}
        +
        \left(\frac{1}{z}+2\alpha\frac{C}{S}\right)g
        -
        \frac{1}{z}g^2.
    \end{equation}
\end{lemma}

\begin{proof}
    We compute
    \begin{equation*}
        \partial_z g = \frac{g}{z}+z\frac{G_{zz}}{G}-\frac{g^2}{z} = \frac{g-g^2}{z}
        +g\left((2\alpha-1)\frac{C}{S}+\frac{B}{Q}\right) = \frac{g-g^2}{z}
        +g\left(
        2\alpha\frac{C}{S}
        +\frac{\alpha z|S|^{2\alpha-2}}{g}
        \right),
    \end{equation*}
    using \(g=zG_z/G\), the quotient rule, \(G_z=\alpha|S|^{2\alpha-2}SQ\), and \(G=SB-CQ\). A last simplification concludes the proof.
\end{proof}

The proof of \cref{global inequality} proceeds by contradiction.
If \(g(z_1,\phi)>M_\alpha\) at some point, then
\cref{lem:boundary-input} shows that, with \(\phi\) fixed,
\(g(z,\phi)<M_\alpha\) whenever \(0<|z|\) is sufficiently small.
Thus, as \(|z|\) increases from zero toward \(|z_1|\) along the same
half-interval, \(g\) has a first contact with the level \(M_\alpha\) at
some \(z\), where
\(g(z,\phi)=M_\alpha\) and
\(z\partial_zg(z,\phi)\geq0\).
Although the signs of \(A,B,S,C,Q\) vary with \((z,\phi)\), the
Riccati equation restricts them at such a first-contact point to the
three relations stated below.

\begin{lemma}[Crossing sign identities]
    Assume that, for some fixed \(\phi\), there exists
    \(z\) with \(0<|z|<\pi_\alpha\) such that $g(z,\phi)=M_\alpha$ and $z\partial_zg(z,\phi)\geq0$.
    Then, it holds that
    \begin{equation}\label{main sign reduction}
        z\frac{C}{S}>0,\qquad
        \frac{A}{Q}<0,\qquad
        z\frac{B}{Q}>0.
    \end{equation}
\end{lemma}

\begin{proof}
    At such a point, the quantities \(G\), \(S\),
    and \(Q\) are nonzero. Indeed, \(G\neq0\) since
    \(0<|z|<\pi_\alpha\), while
    \(S=0\) or \(Q=0\) would imply \(g=0\), contradicting
    \(g=M_\alpha>0\). Evaluating the Riccati equation \cref{Riccati Equation} at
    \(g=M_\alpha\) and using \(z\partial_zg\geq0\) gives
    \begin{equation*}
        2\alpha z\frac{C}{S}
        +
        \frac{\alpha z^2|S|^{2\alpha-2}}{M_\alpha}
        \geq
        M_\alpha-1.
    \end{equation*}
    The level equation \(g=M_\alpha\) reads $M_\alpha(SB-CQ)
        =
        \alpha z|S|^{2\alpha-2}SQ$. Dividing by \(SQ\) and multiplying by \(z/M_\alpha\), we obtain
    \begin{equation*}
        z\frac{B}{Q}
        =
        z\frac{C}{S}
        +
        \frac{\alpha z^2|S|^{2\alpha-2}}{M_\alpha},
        \qquad
        \frac{A}{Q}
        =
        1-z\frac{B}{Q},
    \end{equation*}
    where the second identity uses \(Q=A+zB\). Since
    \(|z|<\pi_\alpha\leq\pi\), \(|S|\leq1\), and
    \cref{Malpha equivalent forms} holds, we find
    \begin{equation*}
        0\leq
        \frac{\alpha z^2|S|^{2\alpha-2}}{M_\alpha}
        \leq
        \frac{\alpha\pi^2}{M_\alpha}
        <
        \frac{12\alpha}{M_\alpha}\leq
        \frac{4(2\alpha+1)(8\alpha+1)}{9M_\alpha}
        \leq
        M_\alpha-1.
    \end{equation*}
    Consequently,
    \begin{equation*}
        z\frac{C}{S}
        \geq
        \frac{1}{2\alpha}
        \left(
        M_\alpha-1
        -\frac{\alpha z^2|S|^{2\alpha-2}}{M_\alpha}
        \right)
        >0.
    \end{equation*}
    Finally, since \cref{Malpha equivalent forms} gives
    \(M_\alpha>2\alpha+1\),
    \begin{equation*}
        z\frac{B}{Q}
        \geq
        \frac{1}{2\alpha}
        \left(
        2\alpha z\frac{C}{S}
        +\frac{\alpha z^2|S|^{2\alpha-2}}{M_\alpha}
        \right)
        \geq
        \frac{M_\alpha-1}{2\alpha}
        >1.
    \end{equation*}
    Hence \(A/Q=1-zB/Q<0\), which proves
    \cref{main sign reduction}.
\end{proof}
Next, we record an estimate that will be used below.
\begin{lemma}\label{tangent excess lemma}
    For \(0<r<\pi_\alpha/2\), one has
    \begin{equation}\label{tangent inequality}
        \tan_\alpha r-r
        >
        \frac{\alpha}{2\alpha+1}r^{2\alpha+1}.
    \end{equation}
\end{lemma}

\begin{proof}
    We compute
    \begin{equation*}
        (\tan_\alpha t-t)'
        =
        \alpha\frac{\sin^{2\alpha}_\alpha t}{\cos_\alpha^2 t}, \quad \text{ and } \quad \left(\frac{\sin_\alpha t}{(\cos_\alpha t)^{1/\alpha}}\right)'
        =
        (\cos_\alpha t)^{-1-1/\alpha}>1,
    \end{equation*}
    for \(0<t<\pi_\alpha/2\). Hence
    \(\sin_\alpha t/(\cos_\alpha t)^{1/\alpha}>t\), and therefore
    \((\tan_\alpha t-t)'>\alpha t^{2\alpha}\). Integrating from \(0\)
    to \(r\) completes the proof.
\end{proof}

The sign relations in \cref{main sign reduction} imply that there is a
unique \(j\in\mathbb Z\) such that
\begin{equation*}
    \min\{\phi,z+\phi\}
    <
    j\pi_\alpha
    <
    \max\{\phi,z+\phi\},
\end{equation*}
or equivalently
\((\phi-j\pi_\alpha)(z+\phi-j\pi_\alpha)<0\). The next lemma shows
that the corresponding coordinates
\begin{equation*}
    r:=|\phi-j\pi_\alpha|,
    \qquad
    s:=|z+\phi-j\pi_\alpha|
\end{equation*}
belong to \((0,\pi_\alpha/2)\) and satisfy \(|z|=r+s\). The expansions in the proofs of \cref{thm:blowup} and
\cref{lem:boundary-input} show that \(G(z,\phi)/z>0\) for
sufficiently small \(z\neq0\). Since \(G(\cdot,\phi)\) does not vanish
for \(0<|z|<\pi_\alpha\), it follows that
\(\operatorname{sgn}G(z,\phi)=\operatorname{sgn}z\) throughout this
domain. In the coordinates
above, this yields the factorization
\begin{equation}\label{crossing factorization}
    |G(z,\phi)|(M_\alpha-g(z,\phi))
    =
    \cos_\alpha(r)\Lambda_\alpha(r,s)E(r,s),
\end{equation}
where
\begin{equation*}
    \Lambda_\alpha(r,s)
    :=
    M_\alpha\cos_\alpha(s)
    +
    \alpha(r+s)\sin_\alpha(s)^{2\alpha-1},
    \qquad
    E(r,s)
    :=
    \tan_\alpha(r)
    +
    \frac{M_\alpha\sin_\alpha(s)}{\Lambda_\alpha(r,s)}
    -(r+s).
\end{equation*}
Recall that both \(\cos_\alpha(r)\) and \(\Lambda_\alpha(r,s)\) are positive on
\((0,\pi_\alpha/2)\). Thus the contradiction hypothesis would force \(E(r,s)=0\), whereas
we will end up proving that \(E(r,s)>0\) on
\((0,\pi_\alpha/2) \times (0,\pi_\alpha/2)\).

\begin{lemma}\label{Crossing Normal Form}
    Assume that \(0<|z|<\pi_\alpha\) is such that $g(z,\phi)=M_\alpha$ and $z\partial_zg(z,\phi)\geq0$.
    Then there exist \(r,s\in(0,\pi_\alpha/2)\), with
    \(|z|=r+s\), such that \(E(r,s)=0\).
\end{lemma}

\begin{proof}
    The first two relations in \cref{main sign reduction} give
    \begin{equation*}
        \operatorname{sgn}A
        =
        -\operatorname{sgn}(z)\operatorname{sgn}B.
    \end{equation*}
    Hence there are unique \(j\in\mathbb Z\) and
    \(r\in(0,\pi_\alpha/2)\) such that
    \(\phi=j\pi_\alpha-\operatorname{sgn}(z)r\). By
    \cref{alpha antiperiodicity},
    \begin{equation*}
        A=-\operatorname{sgn}(z)(-1)^j\sin_\alpha r,
        \qquad
        B=(-1)^j\cos_\alpha r,
        \qquad
        Q
        =
        \operatorname{sgn}(z)(-1)^j\cos_\alpha r
        \bigl(|z|-\tan_\alpha r\bigr).
    \end{equation*}
    The relation \(zB/Q>0\) and \cref{tangent excess lemma} give
    \(|z|>\tan_\alpha r>r\). Thus, setting \(s:=|z|-r\),
    \begin{equation*}
        z+\phi
        =
        j\pi_\alpha+\operatorname{sgn}(z)s,
        \qquad
        |z|=r+s,
        \qquad
        0<s<\pi_\alpha.
    \end{equation*}
    Applying \cref{alpha antiperiodicity} again gives
    \begin{equation*}
        S=\operatorname{sgn}(z)(-1)^j\sin_\alpha s,
        \qquad
        C=(-1)^j\cos_\alpha s.
    \end{equation*}
    The remaining relation in \cref{main sign reduction} becomes
    \(zC/S=|z|\cot_\alpha s>0\), so
    \(s\in(0,\pi_\alpha/2)\). Direct substitution into \(g=zG_z/G\) then gives
    \begin{equation*}
        g(z,\phi)
        =
        \frac{
            \alpha(r+s)\sin_\alpha(s)^{2\alpha-1}
            \bigl(r+s-\tan_\alpha r\bigr)
        }{
            \sin_\alpha s
            -\cos_\alpha s\bigl(r+s-\tan_\alpha r\bigr)
        }.
    \end{equation*}
    This calculation also gives \cref{crossing factorization}.
    Since \(g(z,\phi)=M_\alpha\), rearranging yields
    \begin{equation*}
        r+s-\tan_\alpha r
        =
        \frac{M_\alpha\sin_\alpha s}{\Lambda_\alpha(r,s)},
    \end{equation*}
    which is precisely \(E(r,s)=0\).
\end{proof}

To prove that \(E\) is positive, we use the decomposition
\begin{equation}\label{E decomposition}
    E(r,s)
    =
    \bigl(\tan_\alpha r-r\bigr)
    -
    \left[
        s-
        \frac{M_\alpha\sin_\alpha s}
        {M_\alpha\cos_\alpha s+\alpha(r+s)\sin_\alpha(s)^{2\alpha-1}}
        \right].
\end{equation}
The strict inequality in \cref{tangent excess lemma} controls the first term in \cref{E decomposition} and is essential: it ultimately gives \(E(r,s)>0\), rather than merely \(E(r,s)\geq0\), which closes the contradiction argument. To estimate
the bracketed term, for fixed \(s\in(0,\pi_\alpha/2)\) define
\begin{equation}\label{Psis definition}
    \Psi_s(x)
    :=
    x s^{2\alpha}
    -\frac{\alpha}{2\alpha+1}s^{2\alpha+1}
    -s
    +\frac{\sin_\alpha(s)}
    {\cos_\alpha(s)+x\sin_\alpha(s)^{2\alpha-1}},
    \qquad x\geq0.
\end{equation}
At \(x=\alpha(r+s)/M_\alpha\), the inequality
\(\Psi_s(x)\geq0\) is exactly
\begin{equation}\label{E second term estimate}
    s-
    \frac{M_\alpha\sin_\alpha s}
    {M_\alpha\cos_\alpha s+\alpha(r+s)\sin_\alpha(s)^{2\alpha-1}}
    \leq
    \alpha s^{2\alpha}
    \left(\frac{r+s}{M_\alpha}-\frac{s}{2\alpha+1}\right),
\end{equation}
which is the bound needed for the second term in \cref{E decomposition}.
In the proof of \cref{second deficit estimate lemma}, the
arithmetic--geometric mean inequality will give the lower bound
\(\Psi_s(x)\geq s\mathcal H_\alpha(s)\), where
\begin{equation*}
    \mathcal H_\alpha(s)
    :=
    2\left(\frac{s}{\sin_\alpha s}\right)^{\alpha-1}
    -
    \cos_\alpha s
    \left(\frac{s}{\sin_\alpha s}\right)^{2\alpha-1}
    -
    1
    -
    \frac{\alpha}{2\alpha+1}s^{2\alpha}.
\end{equation*}
We thus first prove that
\(\mathcal H_\alpha(s)\geq0\). When \(\alpha=1\), this reduces to
\(1-s\cot s-s^2/3\geq0\) on \((0,\pi/2)\).

\begin{lemma}\label{H inequality}
    If \(0<s<\pi_\alpha/2\), then \(\mathcal H_\alpha(s)\geq0\).
\end{lemma}

\begin{proof}
    For $\lambda \in (0,1)$, we set
    \[
        R(\lambda)
        :=
        \int_0^1
        \frac{d\tau}{\sqrt{1-\lambda\tau^{2\alpha}}}, \qquad P_\alpha(\lambda)
        :=
        1-\sqrt{1-\lambda}\,R(\lambda)
        -
        \frac{\alpha}{2\alpha+1}
        \lambda R(\lambda)^{\alpha+1}.
    \]
    Taking \(\lambda=\sin_\alpha(s)^{2\alpha}\in(0,1)\), and using
    \cref{alpha energy identity}, we obtain $R(\lambda)
        = s/\sin_\alpha s \geq 1$ and
    \begin{equation} \label{H decomposition}
        \mathcal H_\alpha(s)
        =
        R(\lambda)^{\alpha-1}P_\alpha(\lambda)
        +
        \bigl(R(\lambda)^{\alpha-1}-1\bigr)
        \bigl(1-\sqrt{1-\lambda}\,R(\lambda)^\alpha\bigr).
    \end{equation}
    The proof of \cref{tangent excess lemma} gives
    \(\sqrt{1-\lambda}\,R(\lambda)^\alpha\leq1\). Hence the second term in \cref{H decomposition} is nonnegative, and it remains to prove
    \(P_\alpha(\lambda)\geq0\).

    Differentiating \(R(\lambda)\) and integrating by parts give
    \begin{equation}\label{R differential identity}
        R'(\lambda)
        =
        \frac{
            1-\sqrt{1-\lambda}\,R(\lambda)
        }{
            2\alpha\lambda\sqrt{1-\lambda}
        }.
    \end{equation}
    Expanding the integral at \(\lambda=0\) yields
    \begin{equation*}
        P_\alpha(\lambda)
        =
        \frac{\alpha(\alpha+1)}
        {2(2\alpha+1)^2(4\alpha+1)}\lambda^2
        +O(\lambda^3),
    \end{equation*}
    so \(P_\alpha(\lambda)>0\) for all sufficiently small
    \(\lambda>0\).

    Suppose, for the sake of a  contradiction, that \(P_\alpha(\lambda_1)<0\) for some
    \(\lambda_1\in(0,1)\). Then \(P_\alpha(\lambda)\) has a first zero
    \(\lambda_*\in(0,\lambda_1)\), at which
    \begin{equation*}
        P_\alpha(\lambda_*)=0,
        \qquad
        P_\alpha'(\lambda_*)\leq0.
    \end{equation*}
    Set
    \(k=\sqrt{1-\lambda_*}\,R(\lambda_*)\).
    Since \(P_\alpha(\lambda_*)=0\) and
    \(R(\lambda_*)^2-k^2=\lambda_*R(\lambda_*)^2\),
    \begin{equation}\label{P zero identity}
        \alpha R(\lambda_*)^{\alpha-1}
        \bigl(R(\lambda_*)^2-k^2\bigr)
        =
        (2\alpha+1)(1-k),
        \qquad
        0<k<1<R(\lambda_*).
    \end{equation}
    Differentiating \(P_\alpha(\lambda)\) and using
    \cref{R differential identity} at \(\lambda_*\) gives
    \begin{equation}\label{P derivative at zero}
        P_\alpha'(\lambda_*)
        =
        \frac{R(\lambda_*)}
        {2\sqrt{1-\lambda_*}(2\alpha+1)}
        \left[
            2\alpha+1
            -R(\lambda_*)^{\alpha-1}(\alpha k+\alpha+1)
            \right].
    \end{equation}

    We next prove
    \begin{equation*}
        1-k<\alpha\bigl(R(\lambda_*)^2-1\bigr).
    \end{equation*}
    By \cref{P zero identity}, \(1-k\) is the unique positive root of
    the quadratic equation
    \begin{equation*}
        \alpha R(\lambda_*)^{\alpha-1}t^2
        +
        \bigl(2\alpha+1
        -2\alpha R(\lambda_*)^{\alpha-1}\bigr)t
        -
        \alpha R(\lambda_*)^{\alpha-1}
        \bigl(R(\lambda_*)^2-1\bigr)
        =
        0.
    \end{equation*}
    Its left-hand side is negative at \(t=0\), while at
    \(t=\alpha\bigl(R(\lambda_*)^2-1\bigr)\) it equals
    \begin{equation*}
        \alpha\bigl(R(\lambda_*)^2-1\bigr)
        R(\lambda_*)^{\alpha-1}
        \left[
            \alpha^2\bigl(R(\lambda_*)^2-1\bigr)
            -
            (2\alpha+1)
            \bigl(1-R(\lambda_*)^{1-\alpha}\bigr)
            \right].
    \end{equation*}
    The remaining bracket is positive because
    \begin{align*}
         & \alpha^2\bigl(R(\lambda_*)^2-1\bigr)
        -(2\alpha+1)\bigl(1-R(\lambda_*)^{1-\alpha}\bigr) \\
         & \quad=
        \alpha^2\bigl(R(\lambda_*)-1\bigr)
        \bigl(R(\lambda_*)+1\bigr)
        -(2\alpha+1)(\alpha-1)
        \int_1^{R(\lambda_*)}u^{-\alpha}\,du              \\
         & \quad\geq
        2\alpha^2\bigl(R(\lambda_*)-1\bigr)
        -(2\alpha+1)(\alpha-1)
        \bigl(R(\lambda_*)-1\bigr) =
        (\alpha+1)\bigl(R(\lambda_*)-1\bigr)
        >0,
    \end{align*}
    which proves \(1-k<\alpha\bigl(R(\lambda_*)^2-1\bigr)\). Combining this inequality with \cref{P zero identity}, we obtain
    \begin{equation*}
        \frac{
            2\alpha+1
            -R(\lambda_*)^{\alpha-1}(\alpha k+\alpha+1)
        }{R(\lambda_*)^{\alpha-1}}
        =
        \frac{\alpha\bigl(R(\lambda_*)^2-1\bigr)}{1-k}-1
        >0.
    \end{equation*}
    Equation~\cref{P derivative at zero} now gives
    \(P_\alpha'(\lambda_*)>0\), a contradiction. Therefore
    \(P_\alpha(\lambda)\geq0\) on \((0,1)\).
\end{proof}

Having proved \cref{H inequality}, we now establish that
\(\Psi_s(x)\geq0\) for every \(x\geq0\). Substituting
\(x=\alpha z/M_\alpha\) yields the following estimate; its
specialization to \(z=r+s\) is \cref{E second term estimate}, which
controls the bracketed term in \cref{E decomposition}.

\begin{lemma}\label{second deficit estimate lemma}
    For every \(0<s<\pi_\alpha/2\) and \(x\geq0\), we have $\Psi_s(x)\geq0$.
    In particular, for every \(z>0\), it holds that
    \begin{equation}\label{2nd deficit inequality}
        s-
        \frac{M_\alpha\sin_\alpha s}
        {M_\alpha\cos_\alpha s+\alpha z\sin_\alpha(s)^{2\alpha-1}}
        \le
        \alpha s^{2\alpha}
        \left(\frac{z}{M_\alpha}-\frac{s}{2\alpha+1}\right).
    \end{equation}
\end{lemma}

\begin{proof}
    It suffices to prove the first assertion, since substituting
    \(x=\alpha z/M_\alpha\) into \(\Psi_s(x)\geq0\) gives
    \cref{2nd deficit inequality}.
    Set
    \(y=\cos_\alpha(s)+x\sin_\alpha(s)^{2\alpha-1}>0\).
    By \cref{Psis definition} and the arithmetic--geometric mean
    inequality,
    \begin{align*}
        \Psi_s(x)
         & =
        \frac{s^{2\alpha}}{\sin_\alpha(s)^{2\alpha-1}}y
        +\frac{\sin_\alpha(s)}{y}
        -\cos_\alpha(s)s^{2\alpha}\sin_\alpha(s)^{1-2\alpha}
        -s-\frac{\alpha}{2\alpha+1}s^{2\alpha+1} \\
         & \geq
        2s^\alpha\sin_\alpha(s)^{1-\alpha}
        -\cos_\alpha(s)s^{2\alpha}\sin_\alpha(s)^{1-2\alpha}
        -s-\frac{\alpha}{2\alpha+1}s^{2\alpha+1} \\
         & =
        s\left[
             2\left(\frac{s}{\sin_\alpha(s)}\right)^{\alpha-1}
             -\cos_\alpha(s)
             \left(\frac{s}{\sin_\alpha(s)}\right)^{2\alpha-1}
             -1-\frac{\alpha}{2\alpha+1}s^{2\alpha}
             \right] =s\mathcal{H}_{\alpha}(s)
        \geq0,
    \end{align*}
    where the last inequality follows from \cref{H inequality}.
\end{proof}

Finally, combining \cref{tangent excess lemma} and
\cref{second deficit estimate lemma}, we obtain the strict positivity
of \(E\). The key is the unexpected compatibility between
\cref{tangent inequality,E second term estimate}: after the substitution
\(r=(L-1)s\), the desired positivity reduces exactly to the variational
characterization of \(M_\alpha\) in \cref{Malpha equivalent forms}. It
would be interesting to understand the geometric origin of this
compatibility and whether it can be used to determine sharp curvature
exponents in other models.

\begin{lemma}\label{E cannot have a zero}
    For all
    $0<r,s<\pi_\alpha/2$,
    one has
    $E(r,s)>0$.
\end{lemma}

\begin{proof}
    Set \(L=(r+s)/s>1\). By \cref{E decomposition},
    \cref{tangent inequality}, and \cref{E second term estimate},
    \begin{align*}
        E(r,s)
         & >
        \frac{\alpha}{2\alpha+1}r^{2\alpha+1}
        -\alpha s^{2\alpha}
        \left(\frac{r+s}{M_\alpha}-\frac{s}{2\alpha+1}\right) =
        \alpha s^{2\alpha+1}
        \left[
            \frac{(L-1)^{2\alpha+1}+1}{2\alpha+1}
            -
            \frac{L}{M_\alpha}
            \right]
        \geq0,
    \end{align*}
    where the equality uses \(r=(L-1)s\), and the last inequality
    follows from \cref{Malpha equivalent forms}.
\end{proof}

For clarity, we now combine the preceding lemmas into a proof of
\cref{global inequality}, which, together with
\cref{lem:jacobian-mcp} and the discussion at the beginning of this
section, also proves
\cref{alpha Grushin MCP Theorem}.
\begin{proof}[Proof of \cref{global inequality}]
    Suppose, for contradiction, that
    \(g(z_1,\phi)>M_\alpha\) for some
    \(0<|z_1|<\pi_\alpha\) and \(\phi\in[0,2\pi_\alpha)\).
    By \cref{lem:boundary-input} and
    \cref{Malpha equivalent forms}, the map
    \(t\mapsto g(tz_1,\phi)\) extends continuously to \(t=0\) with a
    value strictly smaller than \(M_\alpha\). It therefore has a first
    contact with the level \(M_\alpha\) at some \(t_*\in(0,1)\).
    Since this level is reached from below,
    \(\frac{d}{dt}g(tz_1,\phi)|_{t=t_*}\geq0\). Thus, at
    \(z=t_*z_1\), one has \(g(z,\phi)=M_\alpha\) and
    \(z\partial_zg(z,\phi)
    =t_*\frac{d}{dt}g(tz_1,\phi)|_{t=t_*}\geq0\).
    \cref{Crossing Normal Form} then produces
    \(r,s\in(0,\pi_\alpha/2)\) such that \(E(r,s)=0\), contradicting
    \cref{E cannot have a zero}. This proves
    \cref{eq:MCPtoprove}.
    Finally, \cref{thm:blowup} shows that the upper bound is
    approached as \((z,\phi)\to(0,0)\). The supremum of \(g\) is
    thus \(M_\alpha\).
\end{proof}

\startappendix
\section{A lower-regularity Jacobian criterion}\label{app:first}
We give a formal proof justifying the application of the
\(\operatorname{MCP}(0,N)\) distortion inequality correspondence in
\cref{MCP theorem Rizzi} to the \(\alpha\)-Grushin plane.

\begin{lemma}[Jacobian criterion for the measure contraction property]
    \label{lem:jacobian-mcp}
    Let \(M\) carry a complete and ideal sub-Riemannian structure generated
    by locally Lipschitz vector fields \(X_1,\ldots,X_k\), and suppose that
    the associated Hamiltonian \(H\) is \(C^2\). Assume that the
    Carnot--Carath\'{e}odory distance induces the manifold topology, and let
    \(\mathfrak m\) be a smooth positive measure on \(M\). For \(q\in M\),
    define the \emph{cotangent injectivity domain} by
    \[
        \operatorname{Inj}^*(q)
        :=
        \left\{
        \lambda\in T_q^*M:
        H(\lambda)>0,\quad
        t_{\operatorname{cut}}(\lambda)>1
        \right\}
    \]
    and let $\operatorname{Inj}(q)
        :=
        \mathrm{exp}_q\bigl(\operatorname{Inj}^*(q)\bigr).$
    Suppose that, for every \(q\in M\), the set
    \(\operatorname{Inj}^*(q)\) is open, the restriction $\mathrm{exp}_q:
        \operatorname{Inj}^*(q)\longrightarrow\operatorname{Inj}(q)$
    is a \(C^1\) diffeomorphism, and
    \[
        \operatorname{Inj}(q)
        =
        M\setminus\bigl(\operatorname{Cut}(q)\cup\{q\}\bigr),
        \qquad
        \mathfrak m\bigl(\operatorname{Cut}(q)\bigr)=0.
    \]
    Write \(d\mathfrak m=\rho\,d\xi\) in local coordinates. Since
    \(s\lambda\in\operatorname{Inj}^*(q)\) for \(0<s\leq1\),
    \(J_q(t;\lambda)\) and \(J_q(1;\lambda)\) have the same sign for every
    \(0<t\leq1\).
    Then, for \(q_1=\mathrm{exp}_q(\lambda)\in\operatorname{Inj}(q)\),
    \[
        \beta_t(q,q_1)
        =
        \frac{
            \rho\bigl(\mathrm{exp}_{q}^t(\lambda)\bigr)
            J_q(t;\lambda)
        }{
            \rho(q_1)
            J_q(1;\lambda)
        },
        \qquad 0<t\leq1.
    \]
    For \(N\geq1\), the inequalities
    \[
        \beta_t(q,q_1)\geq t^N
    \]
    for every \(q\in M\), \(q_1\in\operatorname{Inj}(q)\), and
    \(t\in[0,1]\) hold if and only if \((M,d_{CC},\mathfrak m)\) satisfies
    \(\operatorname{MCP}(0,N)\).
\end{lemma}

\begin{proof}
    Fix \(q\in M\) and \(0<t\leq1\), and define
    \begin{equation}\label{eq:injectivity contraction map}
        \Phi_t^q
        :=
        \mathrm{exp}_{q}^t\circ(\mathrm{exp}_{q})^{-1}
        \colon \operatorname{Inj}(q)\longrightarrow\operatorname{Inj}(q).
    \end{equation}
    The fiberwise quadraticity of \(H\) gives
    \(t\operatorname{Inj}^*(q)\subseteq\operatorname{Inj}^*(q)\), so the
    map in \cref{eq:injectivity contraction map} is well defined and is a \(C^1\)
    diffeomorphism onto its image.

    Let \(q_1=\mathrm{exp}_q(\lambda)\in\operatorname{Inj}(q)\). Since
    \(\operatorname{Inj}(q)\) is open and \(d_{CC}\) induces the manifold
    topology, \(B(q_1,r)\subseteq\operatorname{Inj}(q)\) for every
    sufficiently small \(r>0\). Completeness and ideality give uniqueness of
    the minimizing geodesic from \(q\) to every point of
    \(\operatorname{Inj}(q)\), and hence
    \[
        Z_t(q,B(q_1,r))
        =
        \Phi_t^q(B(q_1,r)).
    \]
    Writing \(\lambda(p)=(\mathrm{exp}_q)^{-1}(p)\), the chain rule and the
    change-of-variables formula give
    \[
        \frac{\mathfrak m(Z_t(q,B(q_1,r)))}
        {\mathfrak m(B(q_1,r))}
        =
        \frac{1}{\mathfrak m(B(q_1,r))}
        \int_{B(q_1,r)}
        \frac{J_q(t;\lambda(p))}{J_q(1;\lambda(p))}
        \rho\bigl(\mathrm{exp}_q^t(\lambda(p))\bigr)\,d\xi(p).
    \]
    Since the integrands are continuous, letting \(r\to0\) yields
    \[
        \beta_t(q,q_1)
        =
        \frac{
            \rho\bigl(\mathrm{exp}_{q}^t(\lambda)\bigr)
            J_{q}(t;\lambda)
        }{
            \rho(q_1)
            J_{q}(1;\lambda)
        }.
    \]

    Suppose now that \(\beta_t(q,p)\geq t^N\) for every
    \(p\in\operatorname{Inj}(q)\). For any Borel set \(B\subseteq M\), another
    application of the change-of-variables formula gives
    \[
        \begin{aligned}
            \mathfrak m(Z_t(q,B))
             & \geq
            \mathfrak m\bigl(\Phi_t^q(B\cap\operatorname{Inj}(q))\bigr) =
            \int_{B\cap\operatorname{Inj}(q)}
            \beta_t(q,p)\,d\mathfrak m(p) \geq t^N\mathfrak m(B),
        \end{aligned}
    \]
    since \(\mathfrak m(M\setminus\operatorname{Inj}(q))=0\). Thus
    \(\operatorname{MCP}(0,N)\) holds.

    Conversely, if \(\operatorname{MCP}(0,N)\) holds, applying it to
    \(B(q_1,r)\) and letting \(r\to0\) in the definition of \(\beta_t\)
    gives \(\beta_t(q,q_1)\geq t^N\).
\end{proof}

For the \(\alpha\)-Grushin plane and every
\(q\in\mathbb G_\alpha^2\), the optimal synthesis recalled in
\cref{sec:alpha-grushin-mcp}, together with the absence of conjugate
times before the cut time, shows that
\(\mathrm{exp}_q:\operatorname{Inj}^*(q)\to\operatorname{Inj}(q)\) is a
\(C^1\) diffeomorphism and that
\(\operatorname{Inj}(q)
=\mathbb G_\alpha^2\setminus
\bigl(\operatorname{Cut}(q)\cup\{q\}\bigr)\). Moreover, the
cut locus has zero Lebesgue measure. The Hamiltonian
\(H(x,y,u,v)=\tfrac12(u^2+|x|^{2\alpha}v^2)\) is \(C^2\) for every
\(\alpha\geq1\), so the exponential map is \(C^1\). Thus
\cref{lem:jacobian-mcp} applies.

\printbibliography[heading=bibintoc]

\end{document}